\documentclass[11pt]{amsart}
\usepackage{amscd}
\usepackage[all]{xy}
\usepackage{graphicx}
\usepackage{amsmath}
\usepackage{soul}
\usepackage{comment}
\usepackage{color}

\usepackage{amsmath, latexsym, amssymb}
\numberwithin{equation}{section}
\theoremstyle{plain}
\newtheorem{lemma}{Lemma}[section]
\newtheorem{proposition}[lemma]{Proposition}
\newtheorem{theorem}[lemma]{Theorem}
\newtheorem{corollary}[lemma]{Corollary}

\theoremstyle{definition}
\newtheorem{definition}[lemma]{Definition}
\newtheorem{remark}[lemma]{Remark}

\begin{document}
\newcommand{\R}{{\mathbb R}}
\newcommand{\C}{{\mathbb C}}
\newcommand{\F}{{\mathbb F}}
\renewcommand{\O}{{\mathbb O}}
\newcommand{\Z}{{\mathbb Z}} 
\newcommand{\N}{{\mathbb N}}
\newcommand{\Q}{{\mathbb Q}}
\renewcommand{\H}{{\mathbb H}}
\renewcommand{\P}{{\mathbb P}}
\newcommand{\Aa}{{\mathcal A}}
\newcommand{\Bb}{{\mathcal B}}
\newcommand{\Cc}{{\mathcal C}}    
\newcommand{\Dd}{{\mathcal D}}
\newcommand{\Ee}{{\mathcal E}}
\newcommand{\Ff}{{\mathcal F}}
\newcommand{\Gg}{{\mathcal G}}    
\newcommand{\Hh}{{\mathcal H}}
\newcommand{\Kk}{{\mathcal K}}
\newcommand{\Ii}{{\mathcal I}}
\newcommand{\Jj}{{\mathcal J}}
\newcommand{\Ll}{{\mathcal L}}    
\newcommand{\Mm}{{\mathcal M}}    
\newcommand{\Nn}{{\mathcal N}}
\newcommand{\Oo}{{\mathcal O}}
\newcommand{\Pp}{{\mathcal P}}
\newcommand{\Qq}{{\mathcal Q}}
\newcommand{\Rr}{{\mathcal R}}
\newcommand{\Ss}{{\mathcal S}}
\newcommand{\Tt}{{\mathcal T}}
\newcommand{\Uu}{{\mathcal U}}
\newcommand{\Vv}{{\mathcal V}}
\newcommand{\Ww}{{\mathcal W}}
\newcommand{\Xx}{{\mathcal X}}
\newcommand{\Yy}{{\mathcal Y}}
\newcommand{\Zz}{{\mathcal Z}}

\newcommand{\zt}{{\tilde z}}
\newcommand{\xt}{{\tilde x}}
\newcommand{\Ht}{\widetilde{H}}
\newcommand{\ut}{{\tilde u}}
\newcommand{\Mt}{{\widetilde M}}
\newcommand{\Llt}{{\widetilde{\mathcal L}}}
\newcommand{\yt}{{\tilde y}}
\newcommand{\vt}{{\tilde v}}
\newcommand{\Ppt}{{\widetilde{\mathcal P}}}
\newcommand{\bp }{{\bar \partial}} 
\newcommand{\ad}{{\rm ad}}
\newcommand{\Om}{{\Omega}}
\newcommand{\om}{{\omega}}
\newcommand{\eps}{{\varepsilon}}
\newcommand{\Di}{{\rm Diff}}
\renewcommand{\Im}{{ \rm Im \,}}
\renewcommand{\Re}{{\rm Re \,}}

\renewcommand{\a}{{\mathfrak a}}
\renewcommand{\b}{{\mathfrak b}}
\newcommand{\e}{{\mathfrak e}}
\renewcommand{\k}{{\mathfrak k}}
\newcommand{\m}{{\mathfrak m}}
\newcommand{\pg}{{\mathfrak p}}
\newcommand{\g}{{\mathfrak g}}
\newcommand{\gl}{{\mathfrak gl}}
\newcommand{\h}{{\mathfrak h}}
\renewcommand{\l}{{\mathfrak l}}
\newcommand{\sm}{{\mathfrak m}}
\newcommand{\n}{{\mathfrak n}}
\newcommand{\s}{{\mathfrak s}}
\renewcommand{\o}{{\mathfrak o}}
\renewcommand{\so}{{\mathfrak so}}
\renewcommand{\u}{{\mathfrak u}}
\newcommand{\su}{{\mathfrak su}}
\newcommand{\X}{{\mathfrak X}}

\newcommand{\ssl}{{\mathfrak sl}}
\newcommand{\ssp}{{\mathfrak sp}}
\renewcommand{\t}{{\mathfrak t }}
\newcommand{\Cinf}{C^{\infty}}
\newcommand{\la}{\langle}
\newcommand{\ra}{\rangle}
\newcommand{\ha}{\scriptstyle\frac{1}{2}}
\newcommand{\p}{{\partial}}
\newcommand{\notsub}{\not\subset}
\newcommand{\iI}{{I}}               
\newcommand{\bI}{{\partial I}}      
\newcommand{\LRA}{\Longrightarrow}
\newcommand{\LLA}{\Longleftarrow}
\newcommand{\lra}{\longrightarrow}
\newcommand{\LLR}{\Longleftrightarrow}
\newcommand{\lla}{\longleftarrow}
\newcommand{\INTO}{\hookrightarrow}

\newcommand{\QED}{\hfill$\Box$\medskip}
\newcommand{\UuU}{\Upsilon _{\delta}(H_0) \times \Uu _{\delta} (J_0)}
\newcommand{\bm}{\boldmath}
\newcommand{\coker}{\mbox{coker}}

\title[Lagrangian submanifolds in  strict  nearly K\"ahler 6-manifolds]{\large Lagrangian submanifolds in  strict  nearly K\"ahler 6-manifolds}
\author[H.V. L\^e and L. Schwachh\"ofer]{H\^ong V\^an L\^e${}^{1}$  and Lorenz Schwachh\"ofer${}^{2}$}
 \date{\today}
 \thanks {H.V.L. is partially supported by RVO: 67985840}
\date{\today}
\medskip
\address{${}^{1}$Institute  of Mathematics of ASCR,
Zitna 25, 11567  Praha 1, Czech Republic} 
\address{${}^{2}$Fakult\"at f\"ur Mathematik,
Technische Universit\"at Dortmund,
Vogelpothsweg 87, 44221 Dortmund, Germany}

\abstract Lagrangian submanifolds  in  strict   nearly   K\"ahler 6-manifolds   are related to    special Lagrangian submanifolds in  Calabi-Yau 6-manifolds  and coassociative    cones  in $G_2$-manifolds. We    prove that    the mean  curvature  of  a Lagrangian  submanifold $L$  in a nearly  K\"ahler  manifold $(M^{2n}, J, g)$ is symplectically dual to  the Maslov  1-form on $L$.   Using relative calibrations, we   derive   a  formula   for the  second variation of  the volume of a  Lagrangian submanifold $L^3$ in  a strict nearly K\"ahler  manifold $(M^6, J, g)$.  This formula implies, in particular,  that any  formal infinitesimal Lagrangian deformation of $L^3$  is  a Jacobi field on $L^3$. We describe   a   finite dimensional  local model of   the moduli space  of compact Lagrangian submanifolds  in a  strict  nearly K\"ahler 6-manifold.  We show that    there is a  real analytic atlas   on $(M^6, J, g)$  in which  the  strict  nearly K\"ahler structure $(J, g)$ is  real analytic.  Furthermore,   w.r.t.  an analytic   strict   nearly K\"ahler structure  the moduli space of Lagrangian submanifolds of $M^6$ is a real analytic variety, whence infinitesimal Lagrangian deformations are smoothly obstructed if and only if they are formally obstructed.
As an application, we relate our results to the description of Lagrangian submanifolds in the sphere $S^6$ with the standard  nearly K\"ahler structure described in \cite{Lotay2012}.

\endabstract 
\maketitle
\tableofcontents


\section{Introduction}\label{sec:intro} 
 Nearly K\"ahler manifolds    first  appeared in Gray's   work \cite{Gray1970} in connection with Gray's notion of weak holonomy. Nearly K\"ahler manifolds  
represent  an important class   in the 16 classes of  almost Hermitian manifolds   $(M^{2n}, J, g)$ classified by Gray and Hervella \cite{GH1980}.
Let us recall    the definition of a nearly K\"ahler manifold $(M^{2n}, J, g)$.    Let $\nabla^{LC}$ denote the Levi-Civita covariant derivative  associated with the Riemannian metric $g$.

\begin{definition}\label{def:nk}(\cite[\S 1, Proposition 3.5]{Gray1970}, \cite{Gray1976}) An almost Hermitian  manifold $(M^{2n}, g, J )$ is called {\it nearly K\"ahler}
if $(\nabla^{LC}_X J )X = 0$ for all $X \in T M^{2n}$.
A nearly K\"ahler manifold is called {\it strict} if we have $\nabla^{LC}_X J \not= 0$ for all $X \in T M^{2n}\setminus \{0\}$, and
it is called of {\it constant type} if for every $x \in M^{2n}$ and $X, Y \in T_x M^{2n}$,
$$||(\nabla^{LC}_X J )Y ||^2 = \lambda^2(||X||^2 ||Y ||^ 2 - \la X, Y\ra ^2- \la  J X, Y\ra^ 2 ),$$
where $\lambda$ is a positive constant.
\end{definition}


\begin{remark}\label{rem:resc} 1.  It is known that any  complete   simply connected nearly K\"ahler  manifold  is a Riemannian product  $M_1\times M_2$  where $M_1$ and 
$M_2$  are K\"ahler   respectively strict nearly K\"ahler  manifolds \cite{Kirichenko1977, Nagy2002}.  Furthermore, a de Rham  type  decomposition   of a strictly nearly K\"ahler  manifold was found by Nagy \cite{Nagy2002a}, where  the factors  of  the decomposition  are    of the following types: 3-symmetric  spaces,  twistor   spaces over quaternionic K\"ahler manifolds of  positive scalar curvature, and strict nearly K\"ahler 6-manifolds.

2. It is easy  to see that if $(M^{2n}, J, g)$ is a  nearly K\"ahler  manifold  of constant  type   $\lambda$, then $(M^{2n}, J, \lambda ^{-1} g)$   and
$(M^{2n}, -J, \lambda^{-1}g)$  are nearly  K\"ahler manifolds  of  constant type  $1$.

3. According to \cite[Theorem 5.2]{Gray1976}, a strict nearly K\"ahler manifold of dimension $6$ is always of constant type.

\end{remark}

On an almost  Hermitian manifold $(M^{2n}, J, g)$   the {\it fundamental  2-form  $\om$},  defined by $\om (X, Y): = g  (JX, Y)$,  measures the connection between  the  almost complex structure $J$  and the Riemannian metric $g$. 
A submanifold $L^n \subset  (M^{2n}, J, g)$  is called {\it Lagrangian}, if $\om|_{L^n} = 0$. 
As  in  symplectic geometry, the  graph  of a  diffeomorphism  of $M^{2n}$ that preserves  $\om$ is a  Lagrangian  submanifold in the almost Hermitian manifold  $(M^{2n} \times M^{2n}, J \oplus (- J), g \oplus g)$.  If  $(M^{2n}, J, g)$  is K\"ahler, then  $\om$ is  symplectic. Lagrangian
submanifolds in K\"ahler manifolds  have been     studied in the context of  calibrated geometry \cite{HL1982}  and  of relative  calibrations \cite{Le1989}, \cite{Le1990},
in the investigation  of the Maslov class \cite{Morvan1981}, \cite{LF1987},   of the variational problem \cite{Le1989}, \cite{Oh1990}, \cite{SW2003}, \cite{SS2010}, and  of the deformation problem/ moduli  spaces \cite{Bryant1987},  \cite{Butscher2003}, \cite{Hitchin1997},  \cite{LO2012}, \cite{McLean1998},  \cite{SS2010},  etc.  The literature on the subject is vast, and the authors    omit   the  name of  many important  papers in the field.

The relation     between   nearly  K\"ahler   manifolds  $(M^{2n},J, g)$ and     Riemannian manifolds  with special  holonomy is best   manifested in   dimension $2n =6$. In dimension 6,  a nearly K\"ahler manifold is either  a K\"ahler manifold or a strict nearly K\"ahler  manifold \cite[Theorem 5.2]{Gray1976}.
It is known from  B\"ar's work \cite{Baer1993}  that  a cone  without singular point  over   a  strict  nearly  K\"ahler manifold $(M^6, J, g)$ is a 7-manifold with $G_2$-holonomy.   It  is not hard to see that  the   cone  over a Lagrangian  submanifold $L^3 $ in a strict nearly K\"ahler manifold $(M^6, J, g)$  is a coassociative cone in  $CM^6$. Thus the study  of strict nearly  K\"ahler   6-manifolds   and their  Lagrangian  submanifolds are  essential for   the study of    singular points of  $G_2$-manifolds   as well as   for the study of singular points  of  coassociative   4-folds. Furthermore,    special Lagrangian  submanifolds  in  Calabi-Yau 6-manifolds   could  be   treated as a limit case of  Lagrangian  submanifolds in nearly K\"ahler  manifolds  when the type constant  $\lambda$ goes to zero   (Remarks \ref{rem:nk6}, \ref{rem:sl}).  We also note that Lagrangian submanifolds in the standard nearly K\"ahler  manifold $S^6$ are found to be intimately  related to holomorphic curves  in $\C P^2$ and they  present  extremely rich  geometry  \cite{DV1996},  \cite{DVV1990}, \cite{Lotay2011}.

In this  paper  we study  Lagrangian  submanifolds  $L^3$ in   strict nearly K\"ahler  6-manifolds $(M^6, J, g)$  in  two aspects:  the  variation of the  volume functional  and Lagrangian deformations    of $L^3$.  Since  $L^3$ are minimal  submanifolds in $(M^6, J, g)$ (Corollary \ref{cor:minnk6}), these   two aspects are related to each other.  In particular, results  from  theory  of minimal submanifolds   are  applicable  to   Lagrangian  submanifolds in  strict nearly K\"ahler 6-manifolds, for instance  see  Remark  \ref{rem:killing}.
To study variation of the   volume functional  of $L^3$  we  apply  the  method  of  relative calibrations  developed  by the  first named   author in \cite{Le1989, Le1990}.
  To study
deformations   of   Lagrangian  submanifolds   in $(M^6, J, g)$ we develop  several  methods.   First  we  reduce  the  overdetermined   equation   for Lagrangian deformations to  an elliptic equation (Proposition \ref{prop:ell}).   Since
 the Fredholm index  of  the  elliptic equation is zero (Proposition \ref{prop:kern})  and, on the other hand,  most interesting   examples  of   Lagrangian  submanifolds   have  nontrivial   deformations, the usual elliptic  method      yields   only limited   results. Thanks  to  our result on the analyticity  of  a strict  nearly  K\"ahler structure (Proposition  \ref{thm:NK-anal}),  we   reduce  the smooth Lagrangian deformation problem to the deformation problem  
in the analytic  category (Proposition \ref{prop:ana}
).  We prove that    the moduli  space  of   Lagrangian deformation is locally  an analytical variety  and hence
an  infinitesimal  Lagrangian deformation is smoothly unobstructed iff  it is  formally unobstructed (Theorem \ref{thm:moduli}, Corollary \ref{cor:formal-tau}). 

Our paper is organized   as follows.  In section \ref{sec:canh} we collect some important results on    the canonical Hermitian connection on nearly K\"ahler manifolds.  Then we  prove  the  existence of a real  analytic structure on any strict  nearly  K\"ahler  6-manifold $(M^6, J, g)$  
in which both $J$ and $g$ are real analytic (Proposition  \ref{thm:NK-anal}).
 In section \ref{sec:strict6}, using a result of the first named author \cite{Le1989},  we  establish  a relation between the Maslov 1-form  and the mean curvature  of a Lagrangian submanifold  in a nearly K\"ahler   manifold $(M^{2n}, J, g)$  (Proposition \ref{prop:maslov}) and show  its consequences (Corollaries \ref{cor:min1}, \ref{cor:minnk6}). 
If  $(M^{2n}, J, g)$ is a strictly   nearly K\"ahler  6-manifold, we derive  a simple formula for the second variation of a Lagrangian submanifold in $(M^6, J, g)$  using relative calibrations (Theorem \ref{thm:jac}). This formula implies, in particular, that any formal infinitesimal  Lagrangian  deformation  is a Jacobi field, generalizing  a result  obtained by McLean for
 special  Lagrangian  submanifolds  (Corollary \ref{cor:harm}, Remarks \ref{rem:h1}, \ref{rem:sl}, \ref{rem:simons}).  

 In section \ref{sec:moduli}  we show that the moduli space of closed Lagrangian submanifolds $L^3 \subset M^6$ of a  strict nearly K\"ahler manifold in the $C^1$-topology is locally a real analytic variety. That is, the set of $C^1$-small deformations of a compact Lagrangian submanifold can be describe as the inverse image of a point of a real analytic map between open domains in a finite dimensional vector space (Theorem \ref{thm:moduli}).
This
 leads us to the following
 \
 
{\bf Conjecture}.  The   group $\Di (M^6, [\om])$  of  all  diffeomorphisms  $g$ of a strictly nearly K\"ahler manifold $(M^6, J, g)$ that  preserves  the fundamental  2-form $\om$ in its conformal class, i.e., $ g^*(\om) = e^f \om$  for   some  $f \in C^\infty(M^6)$, is  a finite  dimensional Lie group.  

Another support  of our conjecture  is our result   that  the group  $\Di (M^6, \om)$  of  all  diffeomorphisms  of a strict nearly K\"ahler manifold $(M^6, J, g)$ that  preserve  the fundamental  2-form $\om$  is a  subgroup  of the  isometry  group  of $(M^6, g)$.

We prove that  any  smooth Lagrangian  deformation  of  a Lagrangian   submanifold $L^3$ in  an analytic  strict nearly K\"ahler  6-manifold   can be written as  a  convergent  power  series (Theorem \ref{thm:moduli}).  As  a result, we look  at   Lagrangian submanifolds  in  standard nearly K\"ahler  sphere  $S^6$   from an  perspective that is different  from  consideration by Lotay in \cite{Lotay2012} (Remark \ref{rem:killing})

\section{Geometry  of   nearly K\"ahler manifolds}\label{sec:canh}
In this section we collect   some important  results on    the canonical Hermitian connection   on 
nearly   K\"ahler  manifolds
(Propositions \ref{cann}, \ref{prop:cannk})  and derive  an important   consequence (Corollary \ref{cor:delta}), which  plays  a central role   in the geometry
of strict nearly K\"ahler 6-manifolds (Proposition \ref{prop:slcal}, Remark \ref{rem:nk6}).
At the end of  this   section we prove the existence  of a real  analytic   structure on $M^6$, in which
both the complex  structure $J$  and the metric  $g$  are  analytic (Proposition \ref{thm:NK-anal}).

\subsection{The canonical Hermitian  connection}\label{subs:canh}


Let $U(M^{2n})$  denote the  principal bundle  consisting of unitary frames $(e_1, Je_1,\cdots, e_n, Je_n)$ over an almost Hermitian manifold $(M^{2n}, J, g)$.  Denote by
$\{e_i ^*,(J e _i) ^*\}$  the dual  frames. Then $\{\theta ^i: = e_i ^* +\sqrt{-1} (Je_i)^*\}$ is the canonical $\C^n$-valued 1-form on $U(M)$. 

Let $\alpha$ be a unitary connection   1-form on  $U(M)$ and  $T$ its torsion 2-form.  The Cartan equation for $\alpha$,   and $T$ \cite[Chapter IX, \S 3]{KN1969} \cite[\S 3]{Le1989}  is  expressed as follows 
$$d\theta ^i = -\alpha ^i _j \wedge \theta ^j +  T^{i}_{ j k} \theta ^j \wedge \theta ^k + T^i_{\bar j k} \bar \theta ^j \wedge \theta ^k + T^i _{\bar j \bar k} \bar \theta ^j \wedge \bar \theta ^k,   $$
$$d \alpha ^i_j = - \alpha ^i _k \wedge \alpha ^k _j + \Om^i _j ,$$
where $\Om$  is the curvature tensor of $\alpha$.
 
\begin{proposition}\label{cann}  (\cite[Chapter IV, \S 112]{Lichnerowicz1955}) Let $(M^{2n}, J, g)$  be  an almost Hermitian manifold. Then there exists a unique  unitary connection  1-form $\alpha $ on  $U(M^{2n})$   such that  its 
torsion  tensor $T$  is a two-form of type $(2,0) + (0,2)$, i.e.,
\[
T(JX, Y) = T(X, JY).
\]
\end{proposition}

We shall denote the Levi-Civita connection of $g$ and the canonical connection from this proposition by $\nabla^{LC}$ and $\nabla^{can}$, respectively. If the almost Hermitian manifold is nearly K\"ahler, then the following is known.

\begin{proposition}\label{prop:cannk} (\cite{Gray1976},  \cite[Theorem 1]{Kirichenko1977}) Suppose that $(M^{2n}, J, g)$  is  a nearly K\"ahler manifold.
\begin{enumerate}
\item Then $T(X, Y) = -J \nabla^{LC} _X(J) Y$.
\item The  associated torsion form  $T^* (X, Y, Z) : =\la  T(X, Y), Z\ra $ is  skew-symmetric. 
\item   $\nabla ^{can}  T^* = 0$.
\end{enumerate}
\end{proposition}

The skew-symmetry of the torsion of the canonical connection of a nearly K\"ahler manifold $(M^{2n}, J, g)$ will play an important r\^ole in our study of $(M^{2n}, J, g)$.

We  shall derive from Proposition  \ref{prop:cannk}  the following

\begin{corollary}\label{cor:delta}
On a nearly K\"ahler  manifold $(M, J, g)$ we have
$d \om (X,Y, Z) = -3T^*(X, Y, JZ)$.  Furthermore, $d\om$ is a 3-form of type $(3, 0) + (0, 3)$, that is,
\[
d \om (JX,Y, Z) = d \om (X,JY, Z) = d \om (X,Y, JZ).
\]
In particular, $\nabla^{can}(d\om) = 0$.
\end{corollary}

\begin{proof} 
 We  use the fact that  the nearly K\"ahler condition is equivalent to the following  condition \cite[Theorem 3.1]{GH1980}
\begin{equation}
X \rfloor d\om = 3 \nabla ^{LC} _X \om \label{eq:dom1}
\end{equation}
for all $X \in TM^{2n}$. A straightforward  calculation shows that (\ref{eq:dom1}) implies
\begin{equation}
d\om (X, Y, Z)= 3 \la  \nabla^{LC}_X (J) Y, Z\ra.\label{eq:dom2}
\end{equation}
Since $T(X, Y) = -J \nabla^{LC} _X(J) Y$, we   obtain  immediately  the first  assertion  of Corollary \ref{cor:delta}.
The second   assertion follows  from  the  first  one, taking into account the fact that $T$ is a 2-form of  type $(2, 0) + (0, 2)$. Finally, the parallelity of $d\om$ follows from Proposition \ref{prop:cannk} (3) and since $\nabla^{can}$ is unitary.
\end{proof}

\subsection{Strict nearly K\"ahler 6-manifolds}\label{sub:nk6}

Among nearly  K\"ahler   manifolds  the class of  strict nearly K\"ahler  6-manifolds  are   most  well-studied. By \cite[Theorem 5.2 (1)]{Gray1976}, any  6-dimensional  nearly K\"ahler manifold,  which is not K\"ahler, is of constant type, i.e., there is a constant $\lambda > 0$ such that 
\begin{equation}
|\nabla_X^{LC} (J) Y | ^2 = \lambda^2\big(| X | ^2 |Y| ^2 - \la X, Y \ra ^2 - \la X, JY \ra ^2\big). \label{eq:calb}
\end{equation}

\begin{proposition}\label{prop:slcal} Assume that $(M^6, g, J)$ is a   strict  nearly  K\"ahler manifold  with constant  type $\lambda$ (cf. (\ref{eq:calb})).
Then $\frac{1}{ 3\lambda}d\om$ is a special Lagrangian calibration.  In particular,  $(M^6, g, J)$ has   an $SU(3)$-structure. 
\end{proposition}

\begin{proof} It is an immediate consequence of (\ref{eq:dom2}) and (\ref{eq:calb}) that $\frac1{3\lambda}d\om$ is a  special Lagrangian calibration, i.e.,  $\om|_\Sigma = 0$ on any $3$-plane $\Sigma \subset T_pM$ with $d\om|_\Sigma = 3 \lambda vol_\Sigma$.

Since $d\om$ is parallel w.r.t. $\nabla^{can}$ and is of type $(3,0) + (0,3)$ by Corollary \ref{cor:delta}, it follows that the complex linear $(3,0)$-form $\Phi \in \Om^3(M, \C)$ given as
\begin{equation} \label{complex-volume}
\begin{array}{rll} \Phi(X, Y, Z) & := & \dfrac1{3 \lambda} \big(d\om(X, Y, Z) - \sqrt{-1}\; d\om(X, Y, JZ)\big)\\ \\
& = & \dfrac1{3 \lambda} \big(d\om(X, Y, Z) + 3 \sqrt{-1}\; T^*(X, Y, Z)\big)
\end{array}
\end{equation}
is parallel w.r.t. $\nabla^{can}$ and nowhere vanishing, so that a strict nearly K\"ahler $6$-manifold carries a canonical $SU(3)$-structure.
\end{proof}

The above argument also shows that for any $3$-dimensional subspace $\Sigma \subset T_pM$ we have (cf. \cite[Chapter III]{HL1982})
\begin{equation} \label{Lagrangian-calibrated}
|\Phi_\Sigma|^2 \leq |vol_\Sigma|^2 \mbox{ with equality if and only if $\om|_\Sigma = 0$.}
\end{equation}

Namely, the first estimate and that equality holds only if $\om_\Sigma = 0$ follows immediately from (\ref{eq:dom2}) and (\ref{eq:calb}). For the converse, let $X,Y,Z$ be an orthonormal basis of a Lagrangian plane. Then (\ref{eq:dom2}) implies that $\nabla_X^{LC}(J) Y$ is orthogonal to $X, Y, JX, JY$, so that $Z \in span (\nabla_X^{LC}(J) Y, J \nabla_X^{LC}(J) Y)$. From this and (\ref{eq:calb}), equality in (\ref{Lagrangian-calibrated}) follows.

Evidently, there are no calibrated submanifolds of $d\om$, since on such a manifold $L \subset M$ we would have $\om|_L = 0$ and hence, $3 \lambda vol_L = d\om|_S = 0$. Thus, by (\ref{Lagrangian-calibrated}) on a Lagrangian submanifold $L \subset M$, $-Im(\Phi)|_L = -\lambda^{-1} T^*|_L$ is a volume form on $L$, i.e., it is calibrated by the non-closed $3$-form $Im(\Phi)$, see also \cite{Le1989}, \cite{Le1990}.

\begin{remark}\label{rem:orient}  For the remainder of our paper we shall choose the orientation on a Lagrangian submanifold $L^3$ such that
\begin{equation} \label{orient-Lagr}
vol_{L^3} = -Im(\Phi)|_{L^3}.
\end{equation}
This   orientation agrees with the 
natural  orientation  of the Lagrangian   sphere $S^3(1) = \H\cap S^6$ in the  standard     strictly nearly K\"ahler    sphere $S^6\subset \Im \O$, see  also subsection \ref{subs:ex} below.  Note that our  choice  of the orientation  of $L^3$ agrees  with that   in   \cite[p. 2309]{Lotay2011}, but  differs  from that in \cite[p. 18]{SS2010}.
\end{remark}

\begin{remark}\label{rem:nk6}  By the above discussion, a nearly  K\"ahler 6-manifold  $(M^6, J, g, \om)$ of   constant  type $\lambda$  satisfies the following equation (cf. \cite[\S 4]{CS2002})
\begin{equation}
d\om = 3 \lambda \, \Re (\Phi) \text{ and }  d\,  \Im  (\Phi)  =-2 \lambda\, \om \wedge \om.\label{eq:csnk6}
\end{equation}
Thus, a  Calabi-Yau 6-manifold can be regarded as an almost strict nearly K\"ahler manifold  with $\lambda = 0$.
\end{remark}

In principle, one could verify (\ref{eq:csnk6}) by a direct calculation, but there is a more elegant way to do this, due to C. B\"ar. Namely, first of all, by rescaling   the  metric $g$ (Remark \ref{rem:resc})  we  can assume that  the  metric is of  constant  type  $\lambda = 1$.

In \cite[\S 7]{Baer1993} B\"ar  constructed   a  3-form  $\varphi$ on the   cone $CM^6= M^6 \times _{r^2} \R^ +$  supplied with     the warped   Riemannian metric
$\bar g = r^2 g +  dr ^2$ over a  strict nearly K\"ahler 6-manifold $(M^6, J, g)$ of constant type $1$.  We identify $M^6$ with
$M^6 \times \{1\} \subset CM^6$.   The form $\varphi$ on $CM^6$ is defined by  \cite[\S 7]{Baer1993}
\begin{equation}
\varphi (r, x) = \frac{r^3}{ 3} d\om + r^2 dr \wedge  \om .\label{eq:phi}
\end{equation}
Since  $d\om$  is of type $(3,0) + (0,3)$  and of comass  3,
for any $x\in M^6$ there is   a local unitary  basis  $((e_1)^*, (Je _1)^*,  \cdots, (Je_3)^*)$   at $T^*_xM^6$  such that
$d\om =  3\Re (dz^1 \wedge dz^2\wedge  dz ^3)$ and $2\om  = - \Im (dz^1 \wedge d \bar  z_1 + dz ^2 \wedge d\bar  z^2 + dz ^3 \wedge  d\bar z^3)$. Here  $dz^i : = (e_i)^* + \sqrt {-1}(Je_i)^*$.
In this basis,  rewriting $dr =  e ^ 7$  and abbreviating $\eps ^{ijk} =  e^i \wedge e^j \wedge  e^k$, we have
\begin{equation}
\varphi (r, x)  = (\eps ^{135}  - \eps ^{146} - \eps ^{236} - \eps ^{245}) + \eps ^{127} + \eps ^{347}  + \eps ^{567}. \label{eq:baer}
\end{equation}
 Clearly, $d\varphi = 0$. B\"ar  also showed that  $d^* \varphi = 0$. Thus $\varphi$  is a 3-form of $G_2$-type  and of comass  1.   In particular,  $\varphi$  (resp. $*\varphi$) is  an associative  (resp.  coassociative) calibration on $CM^6$.  
Furthermore, $d^* \varphi = 0 $ implies  the   second   relation  in (\ref{eq:csnk6})  for $\lambda = 1$. This implies the following result.

\begin{proposition} \label{thm:NK-anal}
Let $(M^6, J, g)$ be a strict nearly-K\"ahler manifold. Then there is a real analytic structure on $M^6$ in which both the complex structure $J$ and the metric $g$ are real analytic. 
\end{proposition}

\begin{proof} It is known that a strict nearly K\"ahler metric on a 6-manifold $M^6$ is an Einstein metric \cite[Lemma 4.8]{Gray1976}. By the DeTurck-Kazdan theorem \cite{DK1981}, $M^6$ possesses an analytic atlas    in which $g$ is an analytic metric. It follows that in the induced real analytic structure on $CM^6$ the aforementioned cone metric $\bar g := dr^2 + r^2 g$ on $CM^6$ is analytic and the vector field $\partial_r$ on $CM^6$ is analytic. Since the form $\varphi \in \Om^3(CM^6, \bar g)$ from (\ref{eq:phi}) defining the $G_2$-structure on $CM^6$ is harmonic, it is analytic as well. Thus, $\p_r \rfloor \phi = r^2 \om$ is analytic, and so is its restriction to the analytic submanifold $M^6 \times \{1\} \subset CM^6$.

Therefore, $\om \in \Om^2(M^6)$ is analytic, and $J$ is defined by contraction of $\om$ with the real analytic metric $g$ and hence analytic as well.
\end{proof}

\section{ Variation of the  volume of Lagrangian   submanifolds}\label{sec:strict6}
In this section    we introduce the notion  of   the Maslov 1-form $\mu(L)$  of a Lagrangian  submanifold $L$ in a  Hermitian  manifold $(M^{2n}, J, g)$
and relate this  notion with the classical  notion  of the Maslov class  of  a Lagrangian  submanifold  in  $(\R^{2n}, \om_0)$ (Remark \ref{rem:maslov}). Then  we prove that $\mu(L)$
is symplectically dual  to the  twice of the mean  curvature $H_L$  of a Lagrangian submanifold   $L$ in a nearly K\"ahler  manifold $(M^{2n}, J, g)$ (Proposition \ref{prop:maslov})  and derive   its consequences (Corollaries \ref{cor:min1}, \ref{cor:minnk6}).  Using relative calibrations, we     prove
a  simple  formula  for the  second  variation  of the  volume  of  a  Lagrangian  submanifold in a strictly nearly K\"ahler 6-manifolds (Theorem \ref{thm:jac})  and discuss
its consequences (Corollary \ref{cor:harm}, Remarks  \ref{rem:h1}, \ref{rem:sl}, \ref{rem:simons}). We discuss   the relation between the obtained results with known results (Remark \ref{rem:maslov1},  \ref{rem:sl}, \ref{rem:simons}). 

\subsection{Maslov  1-form and minimality of a Lagrangian  submanifold  in a nearly K\"ahler manifold}

Let $L$  be a Lagrangian submanifold in  an almost Hermitian manifold $(M^{2n}, J, g)$ and  $(\alpha^i_j)$  the  canonical Hermitian connection  1-form on
$U(M^{2n}, J, g)$. The Gaussian map $g_L$ sends $L$ to the Lagrangian Grassmanian  $Lag(M^{2n})$ of  Lagrangian    subspaces in  the tangent bundle
of $M^{2n}$.  
Denote by $p: U(M^{2n}) \to Lag (M^{2n})$  the projection defined by
$$(v_1, Jv_1, \cdots, v_n, Jv_n) \mapsto  [ v_1 \wedge \cdots  \wedge v_n].$$
Set 
$$\gamma :=-\sqrt{-1} \sum_i \alpha ^i _i. $$  
We recall the  following fact

\begin{lemma}\label{lem:exi1} (cf. \cite{Bryant1987},\cite[Proposition 3.1]{Le1989})    There exists   a 1-form  $\bar \gamma$  on $Lag (M^{2n})$  whose pull-back to  the   unitary frame bundle $U(M^{2n})$ is equal  to   $\gamma$.
\end{lemma}

  We call $2\bar \gamma$ {\it the  universal  Maslov  1-form}  and the induced  1-form $g_L^*(2\bar \gamma)$ on $L$ {\it the Maslov  1-form of $L$}. We  also denote $g_L^*(2\bar \gamma)$ by $\mu (L)$.
  
\begin{remark}\label{rem:maslov} For $M^{2n} = \R^{2n}$   we have  $Lag (M^{2n})  = \R^{2n} \times  U(n)/O(n)$.   In this case  it  is well-known that  the    Maslov 1-form $\mu (L)$  is a closed  1-form and represents   its  Maslov index  of  a Lagrangian  submanifold $L$ \cite{Morvan1981}.  
\end{remark}

Now we  relate the Maslov  1-form   $\mu (L) : = g_L^*(2\bar \gamma)$  with the mean curvature  of  a  Lagrangian submanifold $L$.
We define    a linear isomorphism  $L_\om:   TM \to   T^*M$ as follows.
\begin{equation}
L_\om(V) : = V \rfloor \om.\label{eq:lom}
\end{equation}

\begin{proposition}\label{prop:maslov}  The  Maslov 1-form $\mu (L)$ is   symplectic dual to the minus twice  of the mean curvature $H_L$ of a Lagrangian  submanifold $L$ in a nearly K\"ahler manifold $(M^{2n}, J, g)$. That is,
$$- 2L_\om(H_L) =  \mu (L) .$$
\end{proposition}

\begin{proof} By Proposition \ref{prop:cannk}.2  the 1-form $\sum _{ik} T^i_{\bar i \bar k} \bar\theta ^k$  vanishes,  where $T$ is the   torsion of    the  connection form $\alpha$. Using \cite[Lemmas 2.1, 3.1  and (3.6)]{Le1989}, we obtain for any   normal vector $X$  to $L$
\begin{equation}
\la  -H_L, X \ra   = (\mu (L)/2 , JX) . \label{eq:89}
\end{equation}
Since $\om (-H_L, JX) = \la  -H_L,  X \ra$,  we derive  Proposition \ref{prop:maslov}   immediately from (\ref{eq:89}). 
\end{proof}

Since  the curvature $d\gamma$  form of  the connection form $\gamma$ is the first Chern form of a nearly K\"ahler  manifold we obtain immediately

\begin{corollary}\label{cor:min1}  Assume that  a Lagrangian  submanifold $L$  in a nearly  K\"ahler   manifold $(M, J, g)$ is minimal.
Then  the  restriction  of the first Chern  form  to $L$ vanishes.
\end{corollary}

In the remainder of this  section we    assume  that $L^3$ is a Lagrangian    submanifold in a   strict nearly K\"ahler manifold $(M^6, J, g)$.
We also need  to fix    some notations.
Recall the definition of the $\nabla^{can}$-parallel complex volume form $\Phi = dz_1 \wedge dz_2 \wedge dz_3$ from (\ref{complex-volume}), and let
$$ \alpha : = \Re \Phi = (3 \lambda)^{-1} d\omega,\: \beta : = \Im dz = - \lambda^{-1} T^*.$$

\begin{lemma}\label{lem:hl} Let $\xi$ be a simple 3-vector  in $\R^6 = \C^3$  and $\om$ the standard compatible symplectic form on $\R^6$. Then
\begin{enumerate}
 
\item (\cite[Chapter III Theorem 1.7]{HL1982}) $|\Phi(\xi)|^2  = \alpha (\xi) ^2  +\beta (\xi)^2$.

\item (\cite [Chapter III (1.17)]{HL1982}) $|\Phi(\xi)|^2 + \sum_{i =1}^3 |dz_i \wedge \om(\xi)| ^2 =  |\xi|^2$.
\end{enumerate}
\end{lemma}
We chose the canonical orientation (\ref{orient-Lagr}) on $L$, i.e., $\beta |_{L^3} = - vol_{L^3}$.
For  $x \in L^3$ let $\xi(x)$  denote  the  unit  simple 3-vector  associated  with $T_xL^3$. 
By \cite[Lemma 2.1]{Le1989}, \cite[Lemma 1.1]{Le1990}  for any  $V\in NL^3$  we  obtain  
\begin{equation}
\la  - H_{L^3}, V \ra  = (V \rfloor  d\pm \beta) (\xi).\label{eq:dbeta1}
\end{equation}
(In \cite{Le1989} L\^e showed that the  formula (\ref{eq:dbeta1})  is equivalent to  the  formula (\ref{eq:89}).) 
Using (\ref{eq:csnk6}), we obtain  immediately that $H_{ L^3} = 0$.

\begin{corollary}\label{cor:minnk6}  Any Lagrangian submanifold $L^3$ in a strict nearly K\"ahler  6-manifold  $(M^6, J, g)$  is orientable and minimal.  Hence   its Maslov 1-form vanishes.
\end{corollary}

\begin{remark}\label{rem:maslov1} The relation between   the Maslov class  and  the minimality  of Lagrangian  submanifolds  has been   found for   Lagrangian  submanifolds   in  various  classes of Hermitian  manifolds  \cite {Morvan1981}, \cite{LF1987}, \cite{Le1989}. Corollary \ref{cor:min1}  extends  a  previous  result  by  Bryant \cite[Proposition 1]{Bryant1987} and   partially extends   a  result by  L\^e in \cite[Corollary 3.1]{Le1989}.  The minimality of Lagrangian submanifolds in a strict  nearly K\"ahler  6-manifolds has  been   proved by
 Sch\"afer and Smoczyk  by studying the
second fundamental form  of $L^3$ in $M^6$ \cite[\S 4]{SS2010}, extending a previous  result  by Ejiri \cite{Ejiri}  for  $M^6 = S^6$. The minimality  of   a Lagrangian  submanifold $L$ in a  strict nearly  K\"ahler  manifold $M^6$  can be also  obtained   from    the minimality  of the coassociative  cone $CL^3 \subset  CM^6$.
\end{remark}

\subsection{Second  variation of the volume  of Lagrangian submanifolds}

The second variation  of   the volume of a minimal    submanifold $N$ in a  Riemannian  manifold $M$ has been    expressed by Simons \cite{Simons1968}  in terms of an elliptic 
second order operator $I(N, M)$ that depends  on the second fundamental  form of  $N$ and    the Riemannian curvature on $M$, see also \cite{Le1990b}, \cite{Ohnita1987}. If $L^3$ is a Lagrangian  submanifold  in a strict nearly K\"ahler  manifold $M^6$,  we  shall derive  a simple formula for
 $I(L^3, M^6)$ that depends  entirely on the intrinsic  geometry of $L^3$ supplied with  the induced Riemannian metric.
 
 
\begin{theorem}\label{thm:jac} Assume that  $(M^6, J, g)$  is a strict nearly K\"ahler  manifold  of constant  type $\lambda$.   Let $V$  be  a normal vector field  with compact support on a  Lagrangian  submanifold $L^3 \subset M^6$.
Then the  second  variation of the volume  of  $L^3$  with the variation field $V$  is given  by
\begin{eqnarray}
\frac{d^2}{dt^2}|_{t =0}  vol (L^3_t)=\int_{L^3} \la d (L_\om (V)) - 3\lambda * L_\om (V),  d( L_\om (V)) +\lambda * L_\om (V) \ra\  \nonumber \\
+ \int _{L^3} || d*L_\om (V)||^2 .\nonumber \\
\end{eqnarray}
 
\end{theorem}
\begin{proof}
 Let $\phi_t: L^3 \to M^6$  be a  variation of  $L^3$  generated  by the vector field $V$. 
Set 
$$\xi_t (x): = (\phi_t)_* (\xi(x)).$$
We observe that,  to compute  the  second  variation  of    the  volume  of $L^3$, using Lemma \ref{lem:hl}  and the minimality of $L^3$,  it suffices  to compute  the second variation
of the  integral  over $L^3$  of    $\sum_{i =1} ^3| dz_i \wedge \om  (\xi)| ^2$, $(\alpha ( \xi)) ^2$   and $(\beta(\xi))^2$.
Namely, using  the observation  that for   all  $x \in L^3$
\begin{equation}
|\xi_0(x)| =1 \text {  and } \frac{d}{dt}|_{t =0}   |(\xi_t (x))|  = 0 \label{eq:min}
\end{equation}
  we  obtain

\begin{eqnarray}
\frac{d^2}{dt^2} |_{t =0} vol (\phi_t (L^3)) =  \int _{L^3}\frac{d^2}{dt^2} |_{t =0}|(\xi_t (x)|\, d\, vol_x\nonumber \\
 = \frac{1}{ 2} \int _{L^3}\frac{d^2}{dt^2} |_{t =0}(|\xi_t (x)|^2)\, d\, vol_x.\label{eq:jac1}
\end{eqnarray}

\begin{lemma}\label{lem:mix} For any $x\in L^3$ we have
$$\frac{d^2}{dt^2}|_{t=0} \sum _{i =1} ^3( (dz_i \wedge  \om), \xi_t(x)) ^2 = 2| dL_\om (V) - 3\lambda * L_\om (V)|^2(x). $$
\end{lemma}
\begin{proof} 
Since $\om |_{L^3} = 0$  we have  for all $i$
\begin{equation}
\frac{d^2}{dt^2}|_{t=0} |dz_i \wedge \om (\xi)|^2  =  2[  \frac{d}{dt}|_{ t =0}(dz_i \wedge \om (\xi))]^2. \label{eq:xi2}
\end{equation}
By Proposition \ref{prop:inf1}, taking into account  the rescaling factor $\lambda$,  see also \cite[Theorem 8.1]{SS2010},  we  have
\begin{equation}
\frac{d}{dt}|_{t=0} \phi_t ^* (\om)(x) =d(L_\om (V))(x) - 3(\lambda * L_\om (V))(x).\label{eq:mix1}
\end{equation}
Since the RHS of (\ref{eq:mix1})  is a 2-form  on $L^3$,  there exists  an orthonormal  basis $f^1, f^2, f^3$  of $T^*_x L^3$  and  a   number $c \in \R$ such that
$$d(L_\om (V))(x) - 3(\lambda * L_\om (V))(x) = c\cdot  f^1 \wedge f^2.$$
Using $\om|_{L^3} = 0$    and the expression of  the RHS of (\ref{eq:mix1}) in this basis, we obtain  from (\ref{eq:mix1})
\begin{equation}
\frac{d}{dt}|_{t=0} \sum _{i =1} ^3\phi_t (dz_i \wedge \om) = c  \cdot  f^1 \wedge f^2 \wedge f^3.\label{eq:mix11}
\end{equation}
Using  again $\om|_{L^3} = 0$, we obtain Lemma \ref{lem:mix}  immediately from (\ref{eq:xi2}) and  (\ref{eq:mix11}).
\end{proof}

\begin{lemma}\label{lem:alpha2}  For all $x\in L^3$ we have
$$\frac{d^2}{dt^2}|_{t=0} ( \alpha  (\xi_t(x))^2) = 2|d * L_\om (V)|^2 (x) .$$
\end{lemma}
\begin{proof}   By Proposition \ref{prop:inf1}  we have
\begin{equation}
\frac{dt}{dt} |_{t=0}(\alpha (\xi_t (x)) = (d * L_\om (V))(x) .\label{eq:alpha1}
\end{equation}
Since $\alpha (\xi (x)) = 0$, we  obtain Lemma \ref{lem:alpha2} from  (\ref{eq:alpha1})   immediately.
\end{proof}

\begin{lemma}\label{lem:beta2} We  have
$$\frac{d^2}{dt^2}|_{t =0}\int_{L^3}\beta(\xi_t) ^2 \, dvol_x = 8\lambda \int _{L^3} \la * L_\om (V), d (L_\om (V)) - 3\lambda * L_\om (V)\ra \, d\, vol_x.$$ 
\end{lemma}
\begin{proof} Since    $(V \rfloor \beta)|_{ L^3} = 0$, (see e.g. \cite[Proposition 2.2.(ii)]{Le1989}, \cite[Proposition 1.2.ii]{Le1990},  which is also now called  the  first cousin principle),  using the Cartan formula we have 

\begin{equation}
\frac{d}{dt}|_{t=0}  (\beta(x), \xi_t (x)) =(V \rfloor d\beta , \xi(x)) ,\label{eq:beta1}
\end{equation}
for  all $x \in L^3$.

By (\ref{eq:dbeta1})  the RHS of (\ref{eq:beta1}) vanishes.  Since $\beta (\xi(x)) = -1$  for all $x\in L^3$, we obtain
 
 
\begin{equation}
\frac{d^2}{dt^2}|_{ t =0}(\beta(\xi_t (x)) ^2 )=- 2 \frac{d^2}{dt^2}|_{t=0}(\beta(\xi_t(x)).\label{eq:dtq}
\end{equation}
It follows that 
\begin{equation}
\frac{d^2}{dt^2}|_{t =0}\int_{L^3}\beta(\xi_t(x)) ^2 \, d\, vol_x = -2\frac{d^2}{dt^2}|_{t =0}\int_{L^3}(\phi_t ^*(\beta)  ,\xi)\, d\, vol_x.\label{eq:int1}
\end{equation}
Using the Cartan formula,  we derive from (\ref{eq:int1})
\begin{equation}
\frac{d^2}{dt^2}|_{t =0}\int_{L^3}\beta(\xi_t(x)) ^2 \, d\, vol_x =-2 \int_{L^3}\Ll_V ((V \rfloor d\beta) + d(V \rfloor \beta)).\label{eq:int2}
\end{equation}
Since $\Ll_V (d (V \rfloor \beta)) = d (\Ll_V (V \rfloor \beta)) $,
we obtain  from (\ref{eq:int2}), taking into account  that $d\beta = -2\lambda\, \om \wedge  \om$
\begin{equation}
\frac{d^2}{dt^2}|_{t =0}\int_{L^3}\beta(\xi_t(x)) ^2 \, d\, vol_x =4\lambda \int_{L^3}\Ll_V (V \rfloor (\om \wedge \om)).\label{eq:int3}
\end{equation}
Taking  into account  $V\rfloor (\om\wedge \om) = 2 (V\rfloor \om) \wedge \om$  and $\om|_{L^3} =0$  we obtain  from (\ref{eq:int3})
\begin{equation}
\frac{d^2}{dt^2}|_{t =0}\int_{L^3}\beta(\xi_t(x)) ^2 \, d\, vol_x =8\lambda \int_{L^3}(V\rfloor \om) \wedge \Ll_V (\om).\label{eq:int4}
\end{equation}
Since $(V \rfloor \om) = L_\om (V)$  and $\Ll_V (\om) = dL_\om (V) - 3 \lambda * L_\om (V)$,  we obtain Lemma \ref{lem:beta2} immediately from  (\ref{eq:int4}).
\end{proof}

Now let us complete  the proof  of Theorem \ref{thm:jac}.      
Using Lemma \ref{lem:hl}  we obtain from \ref{eq:jac1}
\begin{eqnarray}
2\frac{d^2}{dt^2} |_{t =0} vol (\phi_t (L^3))= \int_{L^3}\frac{d^2}{dt^2}|_{t =0}\sum_{i =1} ^3(dz_i \wedge  \om, \xi_t) ^2\, d\, vol_x  +\nonumber \\
 \int_{L^3}\frac{d^2}{dt^2}|_{t =0} (\alpha,  \xi_t)^2\, d\, vol_x  + \int _{L^3}\frac{d^2}{dt^2}|_{t =0} (\beta,  \xi_t)^2\, d\, vol_x.\label{eq:jac2}
\end{eqnarray}

Clearly Theorem \ref{thm:jac} follows  from   (\ref{eq:jac2})   and Lemmas \ref{lem:mix}, \ref{lem:alpha2}, \ref{lem:beta2}.
\end{proof}

Using Corollary \ref{cor:ss}, we obtain immediately from Theorem \ref{thm:jac} the following.

\begin{corollary}\label{cor:harm} 1. Any   formal infinitesimal  Lagrangian deformation  with compact  support  of  a  Lagrangian  submanifold $L^3$  in a strict nearly K\"ahler manifold  is a Jacobi field.

 2. Assume that   $L^3$  is   a   compact Lagrangian    submanifold  in a   strict nearly K\"ahler  manifold $(M^6, J, g)$  and $H^1 (L^3,\R) \not = 0$. Let $\beta$ be  a non-zero harmonic  1-form  on  $L^3$. Then  the variation   generated by $L_\om ^{-1} (\beta)$  decreases  the volume of $L^3$.
\end{corollary}

\begin{remark}\label{rem:h1}
There  are  many known  examples   of  Lagrangian  submanifolds  $L^3$  in   the  manifold  $S^6$ supplied with the standard   nearly  K\"ahler structure  induced  from  $\R^ 7 =  \Im \O$ such that $\dim H^1 (L^3)$  is arbitrary large.
For instance, $L^3$ is obtained by  composing the Hopf lifting of a holomorphic  curve $\Sigma_g$ of genus $g$ in the   projective  plane $\C P^2$  to  $S^5$   with  a geodesic   embedding $S^5 \to S^6$    \cite[Theorem 1]{DV1996}, see also \cite[Example 6.11]{Lotay2011}.
\end{remark}

\begin{remark}\label{rem:palmer}   In  \cite{Palmer1998}  Palmer derived  a simple formula  for the second variation  of  Lagrangian submanifolds  in the   standard  nearly K\"ahler 6-sphere by  simplifying  the classical   second  variation
formula  with help  of  (relative) calibrations.
\end{remark}

\begin{remark}\label{rem:sl} Letting  $\lambda$ go to zero, we obtain the   formula  for the second variation   of the volume  of a  special Lagrangian submanifold $L$ in   a Calabi-Yau  manifold $M^6$  with  a  variation field $V$  which is normal to $L$:
\begin{equation}
\frac{d^2}{dt^2}|_{t =0}vol(L_t) =\int_{L} ||d (L_\om V )||^2   + \int _{L} || d^* (L_\om  V)||^2.\label{eq:mclean}
\end{equation}
Formula (\ref{eq:mclean})  has been obtained by McLean in \cite[Theorem 3.13]{McLean1998}    for    special Lagrangian submanifolds in   Calabi-Yau  manifolds of dimension  $2n$  as a  consequence  of his      formula  for the  second  variation of the volume of calibrated  submanifolds, using moving frame method.    Note that our  proof   of Theorem \ref{thm:jac}  can be    easily adapted to give
(\ref{eq:mclean}) for   special  Lagrangian      submanifolds   $L^n \subset  M^{2n}$.  Here  we   use    the  full version of  Lemma \ref{lem:hl} given in \cite[Chapter III, Theorem 1.7, (1.17)]{HL1982}.    The  first summand in RHS of (\ref{eq:mclean}) is the  second variation of the term $ (| \xi| ^ 2- dz (\xi)^2)/2$. The second  summand in the RHS of (\ref{eq:mclean})  is the  second variation  of the term $(\alpha (\xi))^2 /2$. By \cite[(4.11)]{Le1989},  the  second variation   of the term
$\beta(\xi)$ vanishes, if  $M^{2n}$ is  a  Calabi-Yau manifold. This  proves (\ref{eq:mclean}) for     any dimension $n$. Note that (\ref{eq:mclean})  also follows  from 
 Oh's second  variation  formula for  Lagrangian minimal submanifolds  in K\"ahler  manifolds \cite[Theorem 3.5]{Oh1990}. 
\end{remark}

\begin{remark}\label{rem:simons} Using  the strategy of the  proof    of  Theorem \ref{thm:jac},   we   can have  a (new  simple  proof of a)  formula  for the second variation of the volume  of  $\phi$-calibrated  submanifolds  $N^n $  in a    manifold $M^m$ provided  with a  relative  calibration $\phi$  such that    a   generalized version of Lemma \ref{lem:hl} is valid, that     expresses $|\xi| ^2 $   as  a  sum $|\phi (\xi)|^2 + \sum_{i =1} ^k |\alpha _k (\xi)|^2$. 
Generalized  versions of Lemma \ref{lem:hl}  have  been   found for K\"ahler    $2p$-vectors, coassociative  $4$-vectors, ect.  in \cite{HL1982}.  
\end{remark}

\section{Deformations of Lagrangian submanifolds in strict nearly K\"ahler 6-manifolds}\label{sec:moduli}
 
 In this  section we consider the moduli space of closed Lagrangian submanifolds $L^3 \subset M^6$ of a strict nearly-K\"ahler $6$-manifold $M^6$.
 We show that any $C^1$-small Lagrangian deformation of $L^3$ in $M^6$ is a solution of an elliptic first order PDE of Fredholm index 0 (Propositions \ref{prop:ell}, \ref{prop:kern}).
Furthermore,    a  closed Lagrangian  submanifold  $L^3$  is analytic and   any smooth deformation of $L^3$ is analytic (Proposition \ref{prop:ana}).
Moreover,   the  moduli space  of    smooth  Lagrangian   deformations  of $L^3$    locally is a  finite dimensional   analytic  variety, hence  any  formally unobstructed    infinitesimal deformation is
smoothly unobstructed  (Theorem \ref{thm:moduli}, Corollary \ref{cor:formal-tau}).

Our notation on forms will be as follows. By $\Om^*(L)$ we denote smooth differential forms on $L$. If we wish to specify the degree of regularity, we write $\Om^*_{C^k}(L)$ for the space of $C^k$-regular forms.

\subsection{Deformations of  Lagrangian submanifolds}
Let  $L$  be a  submanifold  in  a  Riemannian manifold   $(M, g)$. Then the normal exponential mapping
$Exp_L: NL \to M$ identifies a neighborhood of the $0$-section in $NL$ with a tubular neighborhood  $U(L)\subset M$ of   $L$. With this, $Exp_L$ , which we shall also denote by $Exp$ if no confusion arises, identifies $C^1$-small deformations of $L$ with $C^1$-small section $s : L \to   NL$.

Now  assume that $(M,J, g)$  is a Hermitian  manifold  and  $\om$ is the   associated   fundamental 2-form.
If $L \subset M$  is a Lagrangian submanifold, then the isomorphism $L_\om$ from (\ref{eq:lom}) identifies a covector in  $T^*L$ (resp. a 1-form $\alpha \in \Om^1 (L^3)$) with a vector in $NL$ (resp. a section $s_\alpha \in \Gamma (NL)$).  Since we are interested in Lagrangian deformations of $L \subset M$, we therefore consider the map
\begin{equation} \label{eq:def-F}
F: \Om^1_{C^1}(L) \longrightarrow \Om^2_{C^0}(L), \qquad \alpha \longmapsto (Exp(s_\alpha))^*(\om).
\end{equation}
Evidently, $F(0) = 0$ as $L$ is Lagrangian, and the space of $C^1$-small Lagrangian deformations of $L$ can be identified with a neighborhood of $0 \in F^{-1}(0)$. 


Now  we  shall compute the linearization of $F$ at $\alpha = 0$. We begin with the

\begin{lemma}\label{lem:dom} With the  choice of orientation of $L^3$ from (\ref{rem:orient}) we have
for any $\beta \in \Om^1_{C^0}(L^3)$
\begin{equation}
3 *\beta = -(s_\beta\rfloor d\om)|_{L^3}.\label{eq:betastar}
\end{equation}
\end{lemma}
\begin{proof}  As we have remarked  in  subsection \ref{sub:nk6}, using the same notations, at any  given point $x\in L^3$ there  exists a unitary basis $(e^i, Je^i)$  
such that
$d\om (x) = 3 Re (dz^1 \wedge dz^2 \wedge  dz ^3)$  and $\om  = - Im (dz^1 \wedge d\bar z^1 + dz^2 \wedge d\bar z^2 +  dz ^3 \wedge  d \bar z^3)$. Since  special Lagrangian
planes in $T_xM^6$ are  transitive under $SU(3)$-action \cite{HL1982}, we can assume that   $T_x L^3$ is spanned by  $(Je_1,  Je_2, Je_3)$. Denote by 
 $\{(e_1)^*, (Je_1)^*, (e_2)^*, (Je_2)^*, (e_3)^*, (Je_3)^*\}$    the dual frame to $\{ e^1, Je^1, e^2, Je ^2,\\ e^3, Je ^3\}$.
Since  $\om(x)$, $d\om(x)$ and $T_xL^3$  are invariant under  the action $SO(3)\subset SU(3)\subset  Aut (T_x M^6)$, we can assume   further that $\beta =  c\cdot (Je_1)^*$  for some $c\in \R$. Recall that the orientation  (\ref{rem:orient}) on $L^3$    is defined by
$$Im\, (dz^1\wedge dz^2\wedge dz^3)|_{L^3} = - 3 vol |_{L^3},$$
 i.e. $(Je_1, Je_2, Je_3)$  is an oriented  frame.  Then  we have 
$$3* (Je_1) ^* =  3(Je_2)^*\wedge  (Je_3)^* = - (e_1 \rfloor d\om)|_{L^3} = - (L_\om ^{-1}(Je_1)^*\rfloor d\om)|_{L^3},$$  what is required to prove Lemma \ref{lem:dom}.
\end{proof}

\begin{proposition}\label{prop:inf1}
Let $(\alpha_r)_{r \in (-\epsilon, \epsilon)}$ be a $C^1$-regular variation of $\alpha_0 = 0 \in \Om^1(L^3)$, i.e. such that the map
\[
(- \eps, \eps) \times L^3 \longrightarrow T^*L^3, \qquad (r, x) \longmapsto \alpha_r(x) \in T^*_xL^3
\]
is $C^1$, and let $\dot \alpha_0(x) = \p_r(\alpha_r(x))|_{r=0}$ be the pointwise derivative. Then
\[
\left. \frac d{dr}\right|_{r=0} F(\alpha_r) = d\dot \alpha_0 - 3 *\dot \alpha_0,
\]
whence
\begin{equation} \label{eq:deriv1}
\p F|_0(\beta) = d \beta - 3 * \beta \qquad \mbox{for all $\beta \in \Om^1_{C^1}(\Om)$}.
\end{equation}
\end{proposition}

\begin{proof} We define the $C^1$-map  
\[
D: (- \eps, \eps) \times L^3 \longrightarrow M^6, \qquad (r, x) \longmapsto Exp(s_{\alpha_r})_x =: D_r(x).
\]
Note that $D_0 = Id_{L^3}$ and $dD(\p_r)|_{\{0\} \times L^3} = s_{\dot \alpha_0}$. Also, if we let $\Phi_r$ denote the flow of $\p_r$ on $(- \eps, \eps) \times L^3$, then $D_{r+t} = \Phi_r D_t$, whence by definition,
\begin{eqnarray*}
\left. \frac d{dr}\right|_{r=0} F(\alpha_r) & = & \left. \frac d{dr}\right|_{r=0} D_r^*(\om) = \left.\left(\left. \frac d{dr}\right|_{r=0} \Phi_r (D^*(\om))\right)\right|_{\{0\} \times L^3}\\
& = & \left.({\mathfrak L}_{\p_r} D^*(\om))\right|_{\{0\} \times L^3}
= \left.(\p_r \rfloor D^*(d\om) + d(\p_r \rfloor D^*(\om))\right|_{\{0\} \times L^3}\\
& = & D_0^*(s_{\dot \alpha_0} \rfloor d\om) + d(D_0^*(s_{\dot \alpha_0} \rfloor \om))
= s_{\dot \alpha_0} \rfloor d\om + d\dot \alpha_0.
\end{eqnarray*}

Here, we used Cartan's formula for the Lie derivative as well as the fact that by (\ref{eq:lom}), $s_{\dot \alpha_0} \rfloor \om = \dot \alpha_0$. Now the formula follows since $s_{\dot \alpha_0} \rfloor d\om = - 3 *\dot \alpha_0$ by (\ref{eq:betastar}).
\end{proof}

Recall that the Laplace operator on forms is defined as $\triangle = (d + d^*) ^2$, where $d^* = *d*$ is the adjoint of $d$. Proposition \ref{prop:inf1} yields  immediately the  following Corollary \ref{cor:ss}, which has been obtained  by  Sch\"afer-Smoczyk by a different method.

\begin{corollary}\label{cor:ss} (cf. \cite[Theorem 8.1]{SS2010}) Let $(L^3_r)_{r \in (-\eps, \eps)}$ be a $C^1$-regular family of Lagrangian submanifolds of $M^6$, such that $L^3_0 = L^3$ and $L^3_r = Exp(s_{\alpha_r})$ for a family $(\alpha_r)_{r \in (-\eps, \eps)}$ in $\Om^1(L^3)$. Then the derivative $\beta := \dot \alpha_0 = \p_r \alpha_r|_{r=0}$ is a solution of
\begin{equation}Ê\label{eq:lin-F}
*d\beta - 3\beta = 0.
\end{equation}
In particular, $d^* \beta = 0$  and  $\triangle  \beta = 9 \beta$.
\end{corollary}

We call the map $\p F|_0(\beta)$ from (\ref{eq:deriv1}) the {\em linearization of the equation $F = 0$ at $0$}.
We set
 $$\Om^1_{a} (L^3): = \{ \alpha \in  \Om^1 (L^3)|\, \triangle  (\alpha) = a \cdot \alpha\}.$$
 
All eigenvalues $a$ are nonnegative and the eigenspaces $\Om^1_{a} (L^3)$ are finite dimensional, as $\triangle$ is an elliptic differential operator \cite{Rosenberg1997}.

\begin{lemma} \label{lem:spectrum}
The map $*d: \Om^1(L^3) \to \Om^1(L^3)$ is selfadjoint, and its kernel is $dC^{\infty}(L^3) \oplus \Om^1_0$. Moreover, for each $a > 0$ we have the $L^2$-orthogonal decomposition
\[
\Om^1_a(L^3) \cap \ker d^* = K_{\sqrt a}(L^3) \oplus K_{-\sqrt a}(L^3),
\]
where $K_{\pm \sqrt a}(L^3)$ is the $(\pm \sqrt a)$-eigenspace of $*d$.
\end{lemma}

\begin{proof}
The self-adjointness of $*d$ follows since $d^* = *d*$ and the Hodge-$*$ operator is self-adjoint. Moreover, since $*d = d^* *$, it follows that the image of $*d$ equals the image of $d^*: \Om^2(L^3) \to \Om^1(L^3)$, so that the kernel of $*d$ is the orthogonal complement of this image, which by Hodge decomposition equals $dC^{\infty}(L^3) \oplus \Om^1_0$ as claimed.

Since $*d$ commutes with $\triangle$, it follows that $*d$ preserves $\Om^1_a$ and hence, $\Om^1_a(L^3) \cap \ker d^* $. But on $\ker d^*$ we have $(*d)^2 = d^*d = \triangle$, whence on $\Om^1_a(L^3) \cap \ker d^*$ we have $(*d)^2 = a\; Id$, so that these subspaces can be decomposed into $K_{\sqrt a} \oplus K_{-\sqrt a}$ as claimed.
\end{proof}

It follows from (\ref{eq:deriv1}) that
\begin{equation} \label{def:kernel}
T_{L^3} := \ker \p F|_0 = K_3(L^3).
\end{equation}

The equation $F(\alpha) = 0$ with $F$ from (\ref{eq:def-F}) is an overdetermined equation. In fact,   one of  the technical problems we wish to overcome is the fact that $\p F|_0$ from (\ref{eq:deriv1}) is not an elliptic operator, but only the  restriction of an elliptic first order operator to a subspace as we shall show now.

For this purpose, we extend $F$ by its prolongation $dF$ and add another parameter. Namely, we define
$$C^\infty_0 (L^3) : = \{ f \in C^\infty  (L^3)|\, \int _{L^3}  f\;  dvol = 0\}$$
and the extended map $F$ to
\begin{equation} \label{eq:om1}
\begin{array}{lrll}
\hat F : & \Om ^1 (L^3) \oplus C^\infty_0 (L^3) & \longrightarrow & \Om^1 (L^3) \oplus C^\infty_0 (L^3)\\
& (\alpha, f) & \longmapsto & (*F(\alpha) + df) - \dfrac13 *dF(\alpha).
\end{array}
\end{equation}

\begin{proposition}\label{prop:ell}  $(\alpha, f)$  is  a solution of the equation $\hat F(\alpha, f)=0$ if and only if  $\alpha$ is   a  solution   of the equation $F(\alpha)= 0$  and $f =0$.
\end{proposition}

\begin{proof}
By definition (\ref{eq:om1}), $(\alpha, f)$  is a  solution  of  $\hat F (\alpha , f ) = 0$  iff
\begin{equation}
dF(\alpha) = 0 \qquad \text { and } \qquad F(\alpha) =  -* df. \label{eq:salpha}
\end{equation}
Taking the exterior derivative of the second equation in (\ref{eq:salpha}), this implies that $d*df = 0$ and hence, $\triangle f = 0$.  Thus, $f$ is constant, and since $f\in C^\infty _0 (L^3)$,  we  obtain $f=0$. Therefore, $\alpha$  is a solution of  the equation  $F(\alpha) = 0$. The converse is obvious.
\end{proof}

\begin{remark} From this proof, we can also conclude that
\begin{equation} \label{eq:image-in-T3}
\hat F(\alpha, f) \in T_{L^3} \Leftrightarrow f = 0 \mbox{ and } F(\alpha) = \frac13 d*F(\alpha).
\end{equation}
\end{remark}

It follows from (\ref{eq:deriv1})  and (\ref{eq:om1}) that the  differential 
 $$\p \hat F |_{(0, 0)}: \Om ^1 (L^3)\oplus C^\infty _0(L^3) \longrightarrow \Om ^1 (L^3)\oplus C^\infty _0(L^3)$$
  has the following form
 \begin{equation}
 \p \hat F |_{(0,0)} (\beta, g) = (*d\beta - 3 \beta  + dg , d^* \beta) .\label{eq:dpf1}
 \end{equation}
 
\begin{proposition}\label{prop:kern} $\p \hat F |_{(0,0)}$ is a self adjoint elliptic first order differential operator, and
\[
\ker \p \hat F|_{(0,0)} = T_{L^3}.
\]
That is, $(\beta, g)\in \ker  \p \hat F|_{(0,0)}$ iff $*d\beta = 3 \beta$ and  $g  = 0$.
 \end{proposition}
 \begin{proof} The symbol of $\p \hat F |_{(0,0)}$ is coincides with that of $(*d + d^*)|_{\Om^1(L^3)} + d|_{C_0^\infty(L^3)}$, and it is straightforward to see that the square of the latter operator is $\triangle|_{\Om^1(L^3) \oplus C_0^\infty(L^3)}$. From this, the ellipticity of $\p \hat F |_{(0,0)}$ follows.
 
To see that $\p \hat F |_{(0,0)}$ is self adjoint, let $\beta, \gamma \in \Om^1(L^3)$ and $g, h \in C_0^\infty(L^3)$. Then
 \begin{eqnarray*}
\la \p \hat F |_{(0,0)}(\beta, g), (\gamma, h)\ra_{L^2} & = & \la *d\beta - 3 \beta + dg, \gamma \ra_{L^2} + \la d^*\beta, h \ra_{L^2}\\
& = & \la \beta, d^* * \gamma \ra_{L^2} - 3 \la \beta, \gamma \ra_{L^2} + \la g, d^*\gamma \ra_{L^2} + \la \beta, dh \ra_{L^2}\\
& = & \la \beta, *d \gamma - 3 \gamma + dh \ra_{L^2} + \la g, d^*\gamma \ra_{L^2}\\
& = & \la (\beta, g), \p \hat F |_{(0,0)}(\gamma, h) \ra_{L^2}.
\end{eqnarray*}
To compute the kernel, let $(\beta, g)$ be such that $*d\beta - 3 \beta  + dg = 0$ and $d^*\beta = 0$. Then, applying $d^*$ to the first equation and using the second, it follows that $d^*dg = 0$, whence $\triangle g = 0$, so that $g$ is constant and hence vanishes. Thus, $*d\beta - 3 \beta = 0$, so that $\beta \in T_{L^3} \subset \ker d^*$, and this completes the  proof of Proposition \ref{prop:kern}.
 \end{proof}

 Propositions \ref{prop:ell}, \ref{prop:kern}    imply  that,  the expected  dimension of the  moduli  space  of Lagrangian   submanifolds  is of zero dimension.   
Most interesting  examples   of strict nearly K\"ahler  manifolds    possess non-trivial  symmetry group  and  hence   the   moduli space  of  their
Lagrangian   submanifolds  is  of positive dimension.  The  usual  Fredholm theory  is not  yet developed to deal with    this  phenomenon, see also \cite{Lotay2012} for  a    related  consideration.  Thus we shall     exploit   the  analyticity  of  the  strict  nearly K\"ahler  structure  $(J, g)$  on $M^6$  in  our next consideration.

\subsection{Analyticity  of  Lagrangian deformations  and  its consequences}

As we pointed out in Proposition \ref{thm:NK-anal}, a nearly-K\"ahler manifold $(M^6, g, J, \om)$ is real analytic. As we shall show now, Lagrangian submanifolds are analytic as well. More concretely, we have the following

\begin{proposition}\label{prop:ana}  Let  $L^3$  be a   smooth compact Lagrangian   submanifold in an   analytic strictly nearly K\"ahler  6-manifold  $M^6$.  Then $L^3$ is also an analytic  submanifold  of  $M^6$.
\end{proposition}

\begin{proof} Since any  Lagrangian  submanifold $L^3\subset (M^6, J, g)$  is  a minimal submanifold   in  $(M^6,g)$, the  Morrey  regularity theorem for vector solutions of class $C^1$ of a regular variational problem  \cite{Morrey1954}, \cite{Morrey1958} (see also \cite{Morrey2008}, \cite[IV.2.B]{HL1982})  implies that  $L^3 \subset  M^6$  is a real analytic   submanifold  of $M^6$.
\end{proof}

Note that due to the analyticity of $L^3 \subset M^6$, the normal exponential map
 \[
 Exp: T^* L \stackrel {L_\om} \cong N L \longrightarrow M^6
 \]
is a real analytic diffeomorphism onto a tubular neighborhood of $L^3$.

\begin{definition} \label{def:formal-smooth}
Let $L^3 \subset M^6$ be a closed Lagrangian submanifold.
\begin{enumerate}
\item An element $\alpha \in T_{L^3}$ is called
{\it smoothly unobstructed} or {\it smoothly integrable}, if there is a smooth Lagrangian deformation $s(t): L ^3 \to M^6$ such that $s(0) = L^3$ and $\dot s(0) = \alpha$. Otherwise, $\alpha$ is called {\it smoothly obstructed}. 
\item An element $\alpha \in T_{L^3}$ is called
{\it formally unobstructed} or {\em formally integrable}, if  there  exists  a sequence  $\alpha_1 = \alpha, \alpha_2, \cdots \in \Om^1(L^3)$ such that
the formal power series
\begin{equation} \label{eq:alpha-t}
\alpha_t := \sum_{n=1}^\infty \alpha_n t^n \in \Om^1(L^3)[[t]]
\end{equation}
satisfies
\[
F(\alpha_t) 
= 0 \in \Om^2(L^3)[[t]]
\]
as a formal power series in $t$.
\item We call $L^3$ {\em regular } if every $\alpha \in T_{L^3}$ is formally unobstructed.
\end{enumerate}
\end{definition}

Clearly, if $\alpha \in T_{L^3}$ is smoothly unobstructed, then it is formally unobstructed. Indeed, if $s(t)$ is a smooth Lagrangian deformation with $s(0) = L^3$, then $s(t) = Exp(L_\om(\alpha(t)))$ for some cuve $\alpha(t) \in \Om^1(L^3)$ with $\alpha(0) = 0$ such that $F(\alpha_t) \equiv 0$. Let $\alpha_t := \sum_{n=1}^\infty \alpha_n t^n \in \Om^1(L^3)[[t]]$ be the Taylor series of $\alpha(t)$ at $t = 0$. Then $F(\alpha_t)$ is the Taylor series at $t = 0$ of the function $0 \equiv F(\alpha(t))$ and hence vanishes.

The main purpose of this section is to show the converse: every formally unobstructed element $\alpha \in T_{L^3}$ is smoothly obstructed, and this condition is equivalent to the smooth or formal obstructedness of $\alpha$ w.r.t. an analytic function $\tau: U \to T_{L^3}$ with $U \subset T_{L^3}$ a neighborhood of the origin, i.e. of an analytic function in finitely many variables (cf. Corollary \ref{cor:formal-tau} below). More precisely, we show the following

\begin{theorem}\label{thm:moduli}  The moduli  space  of  closed Lagrangian  submanifolds of a $6$-dimensional nearly-K\"ahler manifold in the $C^1$-topology locally is a finite dimensional analytic variety. 

More concretely, for a closed Lagrangian $L^3 \subset M^6$ there is an open neighborhood $U \subset T_{L^3}$ of the origin and a real analytic map $\tau: U \to T_{L^3}$ with $\tau(0) = 0$ and $\p\tau|_0 = 0$, as well as a $C^\infty$-map $\Phi: L^3 \times U \to M^6$ such that $L^3_\alpha := \Phi(L^3 \times \{\alpha\})$ satisfies:

$L^3_0 = L^3$, and any closed submanifold $L' \subset M^6$ $C^1$-close to $L^3$ is Lagrangian if and only if $L' = L^3_\alpha$ for some $\alpha \in \tau^{-1}(0)$.
\end{theorem}

In order to work towards the proof, we let $\tilde \om := (L_\om \circ Exp)^*(\om) \in \Om^2(T^*L^3)$, so that $F(\alpha) = \alpha^*(\tilde \om)$. We associate to each $\alpha \in \Om^1(L^3)$ the vector field $\xi_\alpha$ on $T^*L^3$ which on each fiber $T_p^*L^3$ is constant equal to $\alpha_p$.

If $\alpha \in \Om_{C^k}^*(L^3)$ is $C^k$-regular, we define for each $p \in L^3$ its $C^k$-norm in $p$ and on $L^3$ as
\begin{equation} \label{eq:Ck-norm}
\|\alpha\|_{C^k;p} := \sum_{|I|\leq k} \|(D^I\alpha)_p\|, \qquad \|\alpha\|_{C^k} := \sup_{p \in L^3} \|\alpha\|_{C^k;p},
\end{equation}
where the sum is taken over all multi-indices $I$ in a coordinate system of $L^3$ which is orthogonal at $p$.

\begin{lemma} \label{lem:omn-pointwise}
There are constants $A, C > 0$ such that for each $C^k$-regular $\alpha \in \Om_{C^k}^1(L^3)$, $k \geq 1$, and each $p \in L^3$ we have
\begin{eqnarray} \label{eq:omn-pointwise}
\left\|\left. \dfrac{d^n}{dt^n} \right|_{t=0} (F(t \alpha))\right\|_{C^{k-1};p} 
& \leq & n! A C^n \|\alpha\|_{C^k;p}^n\\
\label{eq:domn-pointwise} \left\|\left. \dfrac{d^n}{dt^n} \right|_{t=0} (d(F(t \alpha)))\right\|_{C^{k-1};p} 
& \leq & n! A C^n \|\alpha\|_{C^k;p}^n.
\end{eqnarray}

In particular, at each $p \in L^3$, the maps $t \mapsto F(t\alpha)_p \in \Lambda^2 T_p^*L^3$ and $t \mapsto d(F(t\alpha))_p \in \Lambda^3 T_p^*L^3$ are real analytic at $t=0$, and all their derivatives w.r.t. $t$ are $C^{k-1}$-regular $2$-forms and $3$-forms on $L^3$, respectively.
\end{lemma}

\begin{proof} Let us describe this situation in local coordinates. Namely, we let $x = (x^i)$ be analytic coordinates on $L^3$ and let $(x; y) = (x^i; y^r)$ be the corresponding bundle coordinates on $T^*L^3 \to L^3$. With this, we can 
write
\[
\tilde \om = f_{ij} dx^i \wedge dx^j + g_{ir} dx^i \wedge dy^r + h_{rs} dy^r \wedge dy^s,
\]
where the coefficients are analytic functions.

After shrinking this coordinate neighborhood, we may assume that the coefficients of the Riemannian metric $g$ are uniformely bounded and, moreover, by \cite[Proposition 2.2.10]{KP2013}, we may assume that for suitable constants $A_1, C_1 > 0$ we have for $\phi \in \{f_{ij}, g_{ir}, h_{rs}\}$ and any continuous vector field of the form $X = a^r(x) \p_{y^r}$ with arbitrary continuous coefficients $a^r(x)$ the pointwise estimates
\begin{eqnarray}\label{eq:der-coeff-bd}
|X^n (\phi)| \leq n! A_1 C_1^n \|X\|^n.
\end{eqnarray}

Let $\alpha \in \Om_{C^k}^1(L^3)$ be $C^k$-regular, and suppose that $\|\alpha\|_{C^0}$ is sufficiently small, so that its graph is given in these coordinates by $y = \hat \alpha(x)$ for $C^k$-regular functions $\hat \alpha(x) = (\hat \alpha^r(x))_{r=1,2,3}$. Thus, in these coordinates we have
\[
\xi_{\alpha}Ê= \hat \alpha_r(x) \dfrac \p{\p y^r},
\]
and for $|t|$ small we have
\begin{eqnarray*}
F(t\alpha) & = & f_{ij}(x; t \hat \alpha(x)) dx^i \wedge dx^j + t g_{ir}(x; t \hat \alpha(x)) dx^i \wedge d\hat \alpha^r(x)\\ 
& & \qquad \qquad \qquad + t^2 h_{rs}(x; t \hat \alpha(x)) d\hat \alpha^r(x) \wedge d\hat \alpha^s(x)\\ & = & \bigg(f_{ij}(x; t \hat \alpha(x)) + t g_{[ir}(x; t \hat \alpha(x)) \frac{\p \hat \alpha^r}{\p x_{j]}}\\
& & \qquad \qquad \qquad + t^2 h_{rs}(x; t \hat \alpha(x)) \frac{\p \hat \alpha^r}{\p x_{[i}} \frac{\p \hat \alpha^s}{\p x_{j]}}\bigg) dx^i \wedge dx^j.
\end{eqnarray*}
Thus, for fixed $x$, the map $t \mapsto F(t\alpha)_x$ yields an analytic curve in $\Lambda^2 T_x^*L^3$. For the derivatives w.r.t. $t$ of the coefficient functions, note that for $\phi \in \{f_{ij}, g_{ir}, h_{rs}\}$ and $m \in \N$ we have
\[
\left. \dfrac{d^m}{dt^m} \right|_{t=0} \phi(x; t \hat \alpha(x)) = (\xi_\alpha)^m(\phi)_{(x;0)},
\]
whence
\begin{eqnarray*}
\left. \dfrac{d^n}{dt^n} \right|_{t=0} f_{ij}(x; t \hat \alpha(x)) & = & (\xi_\alpha)^n(f_{ij})_{(x;0)}\\
\left. \dfrac{d^n}{dt^n} \right|_{t=0} (t g_{ir}(x; t \hat \alpha(x))) & = & n (\xi_\alpha)^{n-1}(g_{ir})_{(x;0)}\\
\left. \dfrac{d^n}{dt^n} \right|_{t=0} (t^2 h_{rs}(x; t \hat \alpha(x))) & = & n (n-1) (\xi_\alpha)^{n-2}(h_{rs})_{(x;0)}.
\end{eqnarray*}
Thus, for the derivatives we get 
\begin{eqnarray}
\nonumber
\left. \dfrac{d^n}{dt^n} \right|_{t=0} F(t\alpha)_x & = & \bigg((\xi_\alpha)^n(f_{ij})_{(x;0)} + n (\xi_\alpha)^{n-1}(g_{[ir})_{(x;0)} \frac{\p \hat \alpha^r}{\p x_{j]}}\\
\label{eq:omn-local}
& & + n (n-1) (\xi_\alpha)^{n-2}(h_{rs})_{(x;0)} \frac{\p \hat \alpha^r}{\p x_{[i}} \frac{\p \hat \alpha^s}{\p x_{j]}}\bigg) dx^i \wedge dx^j.
\end{eqnarray}

Now (\ref{eq:der-coeff-bd}) implies that
\[
|(\xi_\alpha)^m(\phi)_{(x;0)}| \leq m! A_1 C_1^n \|\alpha_x\|^m
\]
for $\phi \in \{f_{ij}, g_{ir}, h_{rs}\}$. Furthermore, since the coefficients of the metric are bounded, there is a constant $C_2 > 0$ such that at every $x$ and for all $i, r$
\[
\left\| \dfrac{\p \hat \alpha^r}{\p x^i} \right\|_{C^k;x} \leq C_2 \|\alpha\|_{C^{k-1};x}.
\]
Since also $\|dx^i \wedge dx^j\|_{C^k;x}$ is uniformely bounded, and as $\|\alpha_x\| \leq \|\alpha\|_{C^k;x}$ for all $k \geq 0$, (\ref{eq:omn-pointwise}) follows for all $p \in L^3$ parametrized by this coordinate system and for constants $A,C$ depending on these coordinates. But as $L^3$ is compact, it can be covered by finitely many such neighborhoods, whence (\ref{eq:omn-pointwise}) follows for all $p \in L^3$ for uniform constants $A, C > 0$.

The estimate on the derivatives of $d(F(t\alpha))_p = (t\alpha)^*(d\tilde \om)$ follows analogously.
\end{proof}

It follows from Lemma \ref{lem:omn-pointwise} that for each $n \in \N_0$, there are symmetric tensors
\begin{equation} \label{eq:omn-global}
\tilde \om_n: \odot^n \Om^1_{C^1}(L^3) \longrightarrow \Om^2_{C^0}(L^3), \qquad d \tilde \om_n: \odot^n \Om^1_{C^1}(L^3) \longrightarrow \Om^3_{C^0}(L^3),
\end{equation}
satisfying 
\[
\|\la \tilde \om_n; \alpha \ra\|_{C^{k-1}} \leq n! A C^n \|\alpha\|_{C^k} \qquad \mbox{and} \qquad \|\la d \tilde \om_n; \alpha \ra\|_{C^{k-1}} \leq n! A C^n \|\alpha\|_{C^k}
\]
by (\ref{eq:omn-pointwise}) and (\ref{eq:domn-pointwise}), respectively, such that for all $\alpha \in \Om^1_{C^1}(L^3)$ with $\|\alpha_p\| < C^{-1}$ with $C>0$ from (\ref{eq:omn-pointwise}) we have the power series expansion
\begin{equation} \label{eq:Series-section}
F(\alpha) = \sum_{n=1}^\infty \dfrac1{n!} \la (\tilde \om_n); \alpha \ra, \qquad dF(\alpha) = \sum_{n=1}^\infty \dfrac1{n!} \la (d\tilde \om_n); \alpha \ra.
\end{equation}
In particular, $\la (d\tilde \om_n); \alpha \ra = d \la (\tilde \om_n)_p; \alpha \ra$. It is important to point out that for $\alpha \in \Om^1_{C^k}(L^3)$ we have $d \la (\tilde \om_n)_p; \alpha \ra \in \Om^1_{C^{k-1}}(L^3)$, even though it is the exterior differential of $\la (\tilde \om_n)_p; \alpha \ra \in \Om^1_{C^{k-1}}(L^3)$. 

For $\alpha \in \Om^1_{C^1}(L^3)$ the flow of $\xi_\alpha$ on $T^*L^3$ is given by the formula
\[
\Phi^{\xi_\alpha}_t(\beta) = \beta + t \alpha,
\]
whence 
\[
(\Phi^{\xi_\alpha}_t)^*(\tilde \om_{\Phi^{\xi_\alpha}_t(p)}) = F(t \alpha)_p,
\]
so that
\begin{equation} \label{eq:Series-section2}
\la (\tilde \om_n); \alpha \ra = \left. \dfrac{d^n}{dt^n} \right|_{t=0} F(t\alpha) = ((\Ll_{\xi_\alpha})^n (\tilde \om)).
\end{equation}
Since the vector fields $\xi_\alpha, \xi_\beta$ for $\alpha, \beta \in \Om^1_{C^1}(L^3)$ commute, it follows that the symmetrization of $(\tilde \om_n)_p$ is given by
\begin{equation} \label{eq:Series-section3}
\la (\tilde \om_n); \alpha_1, \ldots, \alpha_n \ra = \Ll_{\xi_{\alpha_1}} \cdots \Ll_{\xi_{\alpha_n}}(\tilde \om).
\end{equation}

\

\noindent{\it Proof of Theorem \ref{thm:moduli}.}
Observe that by (\ref{eq:Series-section2}) the description of $\la \tilde \om_n; \alpha \ra$ is given in (\ref{eq:omn-local}) and hence a first oder differential operator, and the same holds for the description of $\la d \tilde \om_n; \alpha \ra$. Thus, the maps $\tilde \om_n$ and $d \tilde \om_n$ from (\ref{eq:omn-global}) extend to a symmetric $n$-linear map
\begin{eqnarray*}
\tilde \om_n: \odot^n L^2_k(T^*L^3) & \longrightarrow & L^2_{k-1}(\Lambda^2 T^*L^3),\\
d\tilde \om_n: \odot^n L^2_k(T^*L^3) & \longrightarrow & L^2_{k-1}(\Lambda^3 T^*L^3),
\end{eqnarray*}
where for a vector bundle $E \to L^3$ we denote by $L^2_k(E)$ the Sobolev space of sections of $E$ with regularity $(2;k)$. By (\ref{eq:omn-pointwise}) we obtain the estimates
\begin{eqnarray*}
\|\la \tilde \om_n; \alpha \ra\|_{L^2_{k-1}(\Lambda^2T^*L^3)} & \leq & n! A C^n \|\alpha\|_{L^2_k(T^*L^3)} \qquad \mbox{and}\\
 \|\la d\tilde \om_n; \alpha \ra\|_{L^2_{k-1}(\Lambda^3T^*L^3)} & \leq & n! A C^n \|\alpha\|_{L^2_k(T^*L^3)}.
\end{eqnarray*}
In particular, for $\alpha \in L^2_k$ with $\|\alpha\|_{L^2_k} < C^{-1}$, the power series $\sum_{n=0}^\infty \dfrac1{n!} \la \tilde \om_n; \alpha \ra$ converges in $L^2_{k-1}(\Lambda^2 T^*L^3)$, thus defining the maps
\[
\begin{array}{rllllll}
F_k: B_{C^{-1}}(0) & \longrightarrow & L^2_{k-1}(\Lambda^2 T^*L^3), & \qquad &  \alpha & \longmapsto & \sum_{n=0}^\infty \dfrac1{n!} \la \tilde \om_n; \alpha \ra\\[2ex]
dF_k: B_{C^{-1}}(0) & \longrightarrow & L^2_{k-1}(\Lambda^3 T^*L^3), & \qquad &  \alpha & \longmapsto & \sum_{n=0}^\infty \dfrac1{n!} \la d\tilde \om_n; \alpha \ra
\end{array}
\]
where $B_{C^{-1}}(0) \subset L^2_k(T^*L^3)$ is the ball centered at $0$. Clearly, $F_k$ and $dF_k$ are analytic maps between Banach spaces in the sense of Definition \ref{def:anal}, and moreover, because of (\ref{eq:Series-section}) and (\ref{eq:Series-section2}), $F_k$ and $dF_k$ extend the maps $F: \Om^1_{C^k}(L^3) \to \Om^2_{C^{k-1}}(L^3)$ and $dF: \Om^1_{C^k}(L^3) \to \Om^3_{C^{k-1}}(L^3)$ from before.

Thus, if we let
\[
L_k^2(T^*L^3 \oplus \R)_0 := \{(\alpha, f) \in L_k^2(T^*L^3 \oplus \R) \mid \int f\; dg = 0\},
\]
then $\hat F$ from (\ref{eq:om1}) extends for all $k \geq 1$ to a map
\[
\begin{array}{lrll}
\hat F_k : & \big(B_{C^{-1}}(0) \subset L_k^2(T^*L^3 \oplus \R)_0 \big) & \longrightarrow & L_{k-1}^2(T^*L^3 \oplus \R)_0\\
& (\alpha, f) & \longmapsto & (*F_k(\alpha) + df) - \dfrac13 *dF_k(\alpha)
\end{array}
\]
which is analytic at $(0,0)$ as $F_k$ and $dF_k$ are analytic. Observe that $\p \hat F_k|_{(0,0)}$ is the extension of $\p \hat F|_{(0,0)}$ which by Proposition \ref{prop:kern} is self adjoint and elliptic. Thus, for all $k \geq 1$
\[
\begin{array}{rclcl}
\ker(\p \hat F_k|_{(0,0)}) = \coker(\p \hat F_k|_{(0,0)}) = T_{L^3}.
\end{array}
\]

Denote by $\pi = \pi_k: L_k^2(T^*L^3 \oplus \R)_0 \to T_{L^3}$ the orthogonal projection which is continuous as $T_{L^3}$ is finite dimensional. Then
\begin{equation}Ê\label{eq:F-proj}
\underline{\hat F}_k := \pi - \hat F_k: \big(B_{C^{-1}}(0) \subset L_k^2(T^*L^3 \oplus \R)_0 \big)\longrightarrow L_{k-1}^2(T^*L^3 \oplus \R)_0
\end{equation}
is again analytic at $(0,0)$, and its differential at $(0,0)$ is an isomorphism. Therefore, the inverse function theorem for analytic maps of Banach spaces (Proposition \ref{prop:invb}) implies that there is an analytic inverse of $\underline{\hat F}_k$:
\[
G_k: (U_k \subset L_{k-1}^2(T^*L^3 \oplus \R)_0) \longrightarrow (V_k \subset L_k^2(T^*L^3 \oplus \R)_0).
\]
Let $U := U_k \cap T_{L^3}$ which is an open neighborhood of the origin and independent of $k$ as $T_{L^3} \subset \cap_{k \geq 1} L_k^2(T^*L^3 \oplus \R)_0)$, and define the map
\begin{equation} \label{eq:def-tau}
\tau: U \longrightarrow T_{L^3}, \qquad \tau(\alpha) := \pi(G_k(\alpha)) - \alpha.
\end{equation}
Then $\tau$ is again analytic, and clearly, $\tau(0) = 0$. Moreover, for $\alpha \in T_{L^3}$ we have
\[
\p \tau|_0(\alpha) = \pi(\p G_k|_{(0,0)}(\alpha)) - \alpha = \pi((\p \underline{\hat F}_k|_{(0,0)})^{-1}(\alpha)) - \alpha = 0,
\]
so that $\p \tau|_0 = 0$.

Observe that for $\alpha \in U$ we have $G_k(\alpha) \in \Om^1_{C^\infty}(L^3)$ for all $k \geq 1$, so that we may omit the subscript $k$. Indeed, the smoothness of $G_k(\alpha)$ follows since $U \subset T_{L^3}$ consists of smooth (in fact analytic) forms; moreover, if $G_k(\alpha) = (\tilde \alpha, \tilde f)$, then $F_k(\tilde \alpha, \tilde f) = \alpha \in T_{L^3}$. As $\hat F_k(\tilde \alpha, \tilde f) = \hat F(\tilde \alpha, \tilde f)$ by the smoothness of $(\tilde \alpha, \tilde f)$, (\ref{eq:image-in-T3}) implies that $\tilde f = 0$ and for $\tilde \alpha = G(\alpha)$ we have
\[
F(G_\alpha) = \dfrac13 d*F(G_\alpha).
\]

Let us define the $C^\infty$-map
\[
\Phi: U_0 \times L^3 \longrightarrow M^6, \qquad (\alpha, p) \longmapsto Exp_p(\xi_{G(\alpha)}),
\]
and let $L^3_\alpha := \Phi(L^3 \times \{\alpha\})$. Evidently, $L^3_0 = L^3$ as $\Phi(0,p) = p$. If $L' \subset M^6$ is a closed submanifold which is $C^1$-close to $L^3$, then we may write $L'$ as the image of $p \mapsto Exp_p(\xi_\beta)$ for some $\beta \in \Om^1_{C^1}(L^3) \subset L^2_1(T^*L^3)$ with $\|\beta\|_{C^1}$ sufficiently small so that $\beta \in V_1$ and hence $\alpha := \underline{\hat F}_1(\beta) \in U_1$. By definition, $\om|_{L'} = F(\beta)$, whence by (\ref{eq:F-proj}) $L'$ is Lagrangian iff $F(\beta) = 0$ iff $\underline{\hat F}_1(\beta) = \pi(\beta) \in T_{L^3}$ iff $\alpha := \pi(\beta) \in U_1 \cap T_{L^3} = U$ and $\beta = G(\alpha)$, and the latter is the case iff $\alpha = \pi(G(\alpha))$ iff $\tau(\alpha) = 0$. That is $L'$ is Lagrangian iff $L' = L^3_\alpha$ for some $\alpha \in \tau^{-1}(0)$ as claimed.
\hfill{\QED}

Theorem \ref{thm:moduli} reduces our consideration of smooth  deformations  of Langrangian submanifolds
to the one in analytic category as follows. Namely, in analogy to Definition \ref{def:formal-smooth} we define

\begin{definition} \label{def:formal-smooth-tau}
Let $L^3 \subset M^6$ be a closed Lagrangian submanifold, and let $\tau: U \to T_{L^3}$ be the real analytic map from Theorem \ref{thm:moduli}.
\begin{enumerate}
\item An element $\alpha \in T_{L^3}$ is called
{\it smoothly unobstructed} or {\it smoothly integrable} w.r.t. $\tau$, if there is a smooth curve $\alpha(t)$ in $T_{L^3}$ such that $\alpha(0) = 0$ and $\dot \alpha(0) = \alpha$, and such that $\tau(\alpha(t)) \equiv 0$. Otherwise, $\alpha$ is called {\it smoothly obstructed}. 
\item An element $\alpha \in T_{L^3}$ is called
{\it formally unobstructed} or {\em formally integrable} w.r.t. $\tau$, if  there  exists  a sequence  $\alpha_1 = \alpha, \alpha_2, \cdots \in T_{L^3}$ such that
the formal power series
\begin{equation}
\alpha_t := \sum_{n=1}^\infty \alpha_n t^n \in R[[t]] \otimes_\R T_{L^3}
\end{equation}
satisfies
\[
\tau(\alpha_t)= 0 \in R[[t]] \otimes_\R T_{L^3}
\]
as a formal power series in $t$.
\end{enumerate}
\end{definition}

\begin{corollary} \label{cor:formal-tau}
Let $L^3 \subset M^6$ be as above. Then for $\alpha \in T_{L^3}$ the following are equivalent.
\begin{enumerate}
\item
$\alpha$ is smoothly obstructed.
\item
$\alpha$ is formally obstructed.
\item
$\alpha$ is smoothly obstructed w.r.t. $\tau$.
\item
$\alpha$ is formally obstructed w.r.t. $\tau$.
\end{enumerate}

In particular, $L^3 \subset M^6$ is regular iff $\tau \equiv 0$.
\end{corollary}

\begin{proof}
According to Theorem \ref{thm:moduli}, a smooth family of Lagrangian submanifolds with $s(0) = L^3$ and $\dot s(0) = \alpha$ must be of the form $s(t) = L^3_{\alpha(t)}$ for a smooth curve $\alpha(t) \in T_{L^3}$ with $\alpha(0) = 0$ and $\dot \alpha(0) = \alpha$ such that $\tau(\alpha(t)) \equiv 0$. That is, $\alpha$ is smoothly obstructed iff it is smoothly obstructed w.r.t. $\tau$.

Likewise, working on the level of formal power series, it follows from Theorem \ref{thm:moduli} that $\alpha$ is formally obstructed iff it is formally obstructed w.r.t. $\tau$.

But as $\tau: U \to T_{L^3}$ is an analytic function in finitely many variables, the equivalence of smooth and formal obstruction of $\alpha \in T_{L^3}$ follows immediately from  Artin's approximation theorem \cite[Theorem 1.2]{Artin1968}.

To show the last statement, let $\tau = \sum_{n=2}^\infty \tau_n$ be the analytic expansion of $\tau$ with $\tau_n \in \odot T_{L^3}Ê\to T_{L^3}$. If $\tau \not \equiv 0$, the there is a minimal $n \geq 2$ such that $\tau_n \neq 0$, so that there is some $\alpha \in T_{L^3}$ such that $\tau_n^\alpha := \tau_n(\alpha, \ldots, \alpha) \neq 0$. Let $\alpha_t := \sum_{n=1}^\infty \alpha_n t^n$ be a formal power series with $\alpha_1 = \alpha$. Then $\tau(\alpha_t) = t^n \tau_n^\alpha \mod t^{n+1}$, so that $\tau(\alpha_t) \neq 0$. That is, $\alpha$ is formally obstructed.

Thus, if all $\alpha \in T_{L^3}$ are formally unobstructed, then $\tau \equiv 0$.
\end{proof}

\begin{corollary}
Each connected component of the moduli space of closed regular Lagrangian submanifolds of $M^6$, equipped with the $C^1$-topology, is an analytic manifold whose tangent space at each $L^3$ may be canonically identified with $T_{L^3}$.
\end{corollary}

\begin{proof} Let $L^3 \subset M^6$ be a regular Lagrangian submanifold, so that the map $\tau: U \to T_{L^3}$ from Theorem \ref{thm:moduli} vanishes identically. Then the map $\Phi: L^3 \times U$ induces an analytic parametrization of all $C^1$-close Lagrangian submanifolds of $L^3$, given by $\alpha \mapsto L^3_\alpha$.
\end{proof}

Let us now describe the analytic expansion of $\tau: U \to T_{L^3}$ from Theorem \ref{thm:moduli}.

\begin{proposition} \label{prop:tau-series}
For $\alpha \in T_{L^3}$, define $\hat \tau_n^\alpha \in \Om^1(L^3)$ for $n \in \N$ recursively by
\begin{equation} \label{eq:formal-solution}
\begin{array}{lll}\hat \tau_1^\alpha & = & \alpha\\[2ex]
\hat \tau_n^\alpha & = & \dfrac{n!}3 \displaystyle{\sum_{r=2}^n \sum_{|I|=n} \dfrac1{r!I!} \left( \left(\p \underline{\hat F}|_{(0,0)}\right)^{-1} * (3 - d) \la \om_r; \hat \tau_{i_1}^\alpha, \ldots, \hat \tau_{i_r}^\alpha \ra\right)}_{\Om^1(L^3)},
\end{array}
\end{equation}
summing over all multi-indices $I = (i_1, \ldots, i_r)$, setting $|I| := i_1 + \ldots + i_r$ and $I! := i_1! \cdots i_r!$. Then
\begin{equation} \label{eq:power-tau}
\tau(\alpha) = \sum_{n=2}^\infty \pi(\hat \tau_n^\alpha) =: \sum_{n=2}^\infty \tau_n^\alpha
\end{equation}
with the orthogonal projection $\pi: \Om^1(L^3) \to T_{L^3}$. In fact,
\begin{equation} \label{eq:tau-n}
\tau_n^\alpha = n! \displaystyle{\sum_{r=2}^n \sum_{|I|=n} \dfrac1{r!I!} \pi \left(* \la \om_r; \hat \tau_{i_1}^\alpha, \ldots, \hat \tau_{i_r}^\alpha \ra\right)}.
\end{equation}
\end{proposition}

\begin{proof}
For $\beta \in \Om^1(L^3)$ and any $k \geq 1$, the map $t \mapsto \underline{\hat F}_k(t \beta)$ from (\ref{eq:F-proj}) is real analytic at $t = 0$ with the expansion
\begin{equation} \label{eq:F-powerseries}
\underline{\hat F}_k(t \beta) = t\; \p \underline{\hat F}_k|_{(0,0)}(\beta) - \sum_{n=2}^\infty \dfrac{t^n}{3n!} *(3 - d) \la \tilde \om_n; \beta, \ldots, \beta\ra.
\end{equation}
using the definition of $\hat F$ in (\ref{eq:om1}) and the expansions in (\ref{eq:Series-section}). Since $\p \underline{\hat F}_k|_{(0,0)}$ is invertible, it follows that the map $t \mapsto G_k(t \alpha)$ is also analytic at $t=0$, and we expand it as a power series for an element $\alpha \in T_{L^3}$ as
\begin{equation} \label{eq:F-1-powerseries}
G_k(t \alpha) = \sum_{n=1}^\infty \dfrac{t^n}{n!} \la g_n; \alpha \ldots \alpha \ra =: \sum_{n=1}^\infty \dfrac{t^n}{n!} g_n^\alpha,
\end{equation}
for $n$-multilinear maps $g_n: \odot^n T_{L^3} \longrightarrow \Om^1(L^3) \oplus C^\infty(L^3)_0$ which we decompose as
\[
g_n^\alpha = \hat \tau_n^\alpha + f_n^\alpha.
\]

Applying $\underline{\hat F}$ to this equation, and using (\ref{eq:F-powerseries}), it follows that $g_n^\alpha$ must be solutions of the equation
\begin{eqnarray} \label{eq:Taylor-inverse}
\nonumber t \alpha & = & t\; \p \underline{\hat F}_k|_{(0,0)}(g_1^\alpha) + \sum_{n=2}^\infty t^n \bigg( \dfrac{1}{n!} \p \underline{\hat F}_k|_0(g_n^\alpha)\\
& & \qquad \qquad \qquad - \dfrac 13 \sum_{r=2}^n \sum_{|I| = n} \dfrac1{r! I!} *(3 - d) \la \om_r; \hat \tau_{i_1}^\alpha, \ldots, \hat \tau_{i_r}^\alpha \ra\bigg).
\end{eqnarray}
Comparing the $t$-coefficient and using that $\p \underline{\hat F}_k|_{(0,0)}$ is the identity on $T_{L^3}$, it follows that $g_1^\alpha = \alpha$, whence $\hat \tau_1^\alpha = \alpha$ and $f_1^\alpha = 0$, and comparing the $t^n$-coefficients for $n > 1$ yields 
\begin{eqnarray*}
g_n^\alpha  & = & \dfrac{n!}3 \sum_{r=2}^n \sum_{|I|=n} \dfrac1{r!I!} \left(\p \underline{\hat F}|_{(0,0)}\right)^{-1} *(3 - d) \la \om_r; \hat \tau_{i_1}^\alpha, \ldots, \hat \tau_{i_r}^\alpha \ra.
\end{eqnarray*}
That is, $\hat \tau_n^\alpha$ from (\ref{eq:formal-solution}) is the $\Om^1(L^3)$-component of $g_n^\alpha$.

Now by (\ref{eq:def-tau}), the series expansion of $\tau$ is
\begin{eqnarray*}
\tau(t \alpha) = \pi(G_k(t\alpha)) - t \alpha = \sum_{n=1}^\infty t^n \pi(g_n^\alpha) - t \alpha = \sum_{n=2}^\infty t^n \pi(\hat \tau_n^\alpha),
\end{eqnarray*}
using that $g_1^\alpha = \alpha$. This shows (\ref{eq:power-tau}), and (\ref{eq:tau-n}) follows as $\pi$ is the projection onto the $(+1)$-eigenspace of $\underline{\hat F}_k|_{(0,0)}$ by definition, so that $\pi \left(\underline{\hat F}_k|_{(0,0)}\right)^{-1} = \pi$.
\end{proof}

\begin{definition} Let $L^3 \subset M^6$ be a closed Lagrangian submanifold. The {\em Kuranishi map of $L^3$} is the symmetric bilinear map
\[
K: T_{L^3} \times T_{L^3}Ê\longrightarrow \Om^2(L^3), \qquad K(\alpha_1, \alpha_2) := \la \om_2; \alpha_1, \alpha_2\ra = \Ll_{\xi_{\alpha_1}} \Ll_{\xi_{\alpha_2}}(\om).
\]
\end{definition}

Thus, by (\ref{eq:tau-n}) we have
\[
\tau_2^\alpha = \pi *K(\alpha, \alpha),
\]
whence the we obtain the following result.

\begin{proposition}\label{prop:obstr} 
Assume that $L^3 \subset M^6$ is a Lagrangian submanifold, and let $\alpha \in T_{L^3}$.
If $\alpha$ is smoothly unobstructed, then $\pi *K(\alpha, \alpha) = 0$, i.e.,
\[
\int_{L^3} K(\alpha, \alpha) \wedge \beta = 0 \qquad \mbox{for all $\beta \in T_{L^3}$.}
\]
\end{proposition}

\begin{proof} If $\alpha \in T_{L^3}$ is smoothly unobstructed, then by Theorem \ref{thm:moduli} there is a curve $\alpha(t)$ in $T_{L^3}$ with $\alpha(0) = 0$ and $\dot \alpha(0) = \alpha$ such that $\tau(\alpha(t)) \equiv 0$. As $\p \tau|_0 = 0$, we have
\[
0 = \left. \dfrac{d^2}{dt^2}\right|_{t=0} \tau(\alpha(t)) = 2 \tau_2(\dot \alpha, \dot \alpha)|_{t=0} = 2 \tau_2^\alpha = 2 \pi *K(\alpha, \alpha).
\]
From this, the claim follows.
\end{proof}

Evidently, with increasing $n$, the the $n$-th order formal obstructions of an element $\alpha \in T_{L^3}$ become increasingly involved.

\section{Examples}\label{subs:ex}
 In this section we wish to apply our results to Lagrangian submanifolds of the standard nearly  K\"ahler  sphere $(S^6, J_0, g_0)$ and put our work into the context of the deformation results in \cite{Lotay2012}.
 
Let $\O$ denote the octonians, which is the unique $8$-dimensional normed division algebra. It may be orthogonally decomposed into $\O = \R \cdot 1 \oplus \Im \O$, and there is a vector cross product $\times$ on $\Im \O$, defined as the imaginary part of the octonian multiplication, i.e.
\[
x \times y = - \la x; y\ra + x \times y \qquad \mbox{ for all $x, y \in \O$},
\]
where $\la.;.\ra$ denotes the scalar product on $\O$. Then the automorphism group $G_2$ of $\O$ preserves the inner product and acts on the $7$-dimensional space $\Im \O$, and it clearly preserves the cross product $\times $ on $\Im \O$. Furthermore, we define the $3$-form $\varphi$ on $\Im \O$ by
\[
\varphi(x, y, z) := \la x \times y; z\ra,
\]
so that $\varphi$ is invariant under the action of $G_2$; indeed, $G_2$ can also be described as the group of automorphisms on $\Im\O$ preserving $\phi$.

Let $S^6 \subset \Im \O$ denote the unit sphere with the round metric $g_0$ induced by the inner product $\la.;.\ra$ on $\Im \O$. Then there is an orthogonal almost complex structure $J_0$ on $S^6$, defined as
$$J_0|_p (u) := p \times u $$
for  $p \in S^ 6$  and  $u \in T_p  S^6 \subset  \Im \O$. Since the cone metric over $S^6$ is the flat metric on $\Im \O$ which is clearly a (flat) $G_2$-manifold, the result of B\"ar \cite[\S 7]{Baer1993} already mentioned in section \ref{sub:nk6} implies that $(S^6, g_0, J_0)$ is a strict nearly-K\"ahler manifold. It follows that the action of $G_2$ on $S^6 \subset \Im \O$ preserves the nearly-K\"ahler structure and is in fact the invariance group of this structure. Moreover, this action of $G_2$ on $S^6$ is transitive, with stabilizer $SU(3) \subset G_2$, whence we may write the sphere as a homogeneous space
\[
S^6 = G_2/SU(3).
\]

We call a Lagrangian submanifold $L^3 \subset S^6$ {\em linearly full}, if it is not contained in a totally geodesic sphere $S^5 \subset S^6$. For instance (cf. \cite[Example 6.11.]{Lotay2011}), if $\Sigma \subset \C\P^2$ is a holomorphic curve, then the inverse image of $\Sigma$ under the Hopf fibration $S^5 \to \C\P^2$ yields a Lagrangian submanifold $L^3_\Sigma \subset S^5 \subset S^6$ which is not linearly full. In fact, any Lagrangian submanifold $L^3 \subset S^6$ which is not linearly full is of this type \cite[Theorem 1.1.]{Lotay2011}.

If $\Sigma \subset \C\P^2$ is a curve of degree $d=1$, then the Hopf lift is a totally geodesic $3$-sphere which is preserved by $SO(4) \subset G_2$, whence the space $\Mm_{d=1}$ of all $L_{\Sigma}$ with $\Sigma \subset \C\P^2$ of degree one is diffeomorphic to the symmetric space $G_2/SO(4)$.

For $d=2$, the sets of smooth conics in $\C\P^3$ is the homogeneous space $Sl(3, \C)/SO(3, \C)$, whence the moduli space of the Hopf lifts of curves of degree $d = 1, 2$ is given as
\[
\Mm_{d=2} = G_2 \times_{SU(3)} Sl(3, \C)/SO(3, \C).
\]
In particular, $\Mm_{d=2}$ is a manifold of dimension $16$.

In \cite{Mashimo1985}, Mashimo gave a complete classification of homogeneous Lagrangian sumbmanifolds $L^3 \subset S^6$, i.e., $L^3$ is the orbit of some subgroup $H \subset G_2$. Indeed, there are, up to $G_2$-equivalence, five inequivalent Lagrangian submanifolds; the description of the induced metric is given in \cite{DVV1990}.

\begin{enumerate}
\item
The totally geodesic Lagrangian sphere $L_0^3 := S^3 \subset S^6$, given as the intersection of $S^6$ with a coassociative subspace $V^4 \subset \O$, i.e., such that $\varphi|_V \equiv 0$.
\item
The ``squashed sphere''
\[
L^3_1 := \left\{ \frac{\sqrt 5}3 q i \bar{q} + \frac23 \bar{q} \epsilon \; : \; q \in Sp(1)\right\},
\]
using the decomposition $\O = \H \oplus \H \epsilon$ for some unit octonian $\epsilon \in \H^\perp$. Clearly, $L^3_1$ is again a sphere, and the metric induced by this embedding is a Berger metric, invariant under $U(2) \subset G_2$. That is, every oriented isometry of $L^3_1$ extends to $S^6$.
\item
The space $L_2^3 := L_\Sigma^3$ with the notation from above, where $\Sigma \subset \C\P^2$ is the quadric
\[
z_1^2 + z_2^2 + z_3^2 = 0.
\]
Then $L^3_2 \subset S^5(1) \subset S^6(1)$ is not linearly full and diffeomorphic to $\R\P^3$. In fact, it is acted on simply transitively by the subgroup $SO(3) \subset SU(3) \subset G_2$, where $SU(3)$ is the stabilizer of $(S^5)^\perp$. The induced metric on $L^3_2$ is again a Berger metric, but the only oriented isometries which extend to $S^6$ are the elements of $SO(3)$.
\item
There is a (unique) subgroup $SO(3) \subset G_2$ which acts irreducibly on $\Im\O$, thus identifying $\Im\O$ with $\Hh^3(\R^3)$, the space of harmonic cubic polynomials in the three variables $x,y,z$, as an $SO(3)$-module.

If $p \in \Hh(\R^3)$ is completely reducible, then -- up to a multiple -- it is contained in the $SO(3)$-orbit of one of the two non-equivalent polynomials
\[
p_3(x,y,z) = x(x^2-3y^2) \qquad \mbox{or} \qquad p_4(x,y,z) := xyz,
\]
and we let $L_k^3 := SO(3) \cdot \frac{p_k}{\|p_k\|} \subset S^6$ for $k = 3,4$ be the $SO(3)$-orbit of these polynomials. Then $L_3^3 = SO(3)/D_3$ and $L_4^3 = SO(3)/A_4$, where $D_3$ and $A_4$ denote the dihedral and the tetrahedral group, respectively.

The induced metric on $L_3^3$ is a Berger metric, whereas the induced metric on $L_4^3$ is a metric of constant curvature.
\end{enumerate}

In \cite{Lotay2012}, Lotay calculated the formal tangent spaces $T_{L^3_i}$ for all these examples $L_i^3 \subset S^6$, $i = 0, \ldots, 4$. Furthermore, he discussed the regularity of these spaces, using an approach different from ours which heavily uses the $G_2$-invariance of the nearly K\"ahler structure on $S^6$. His results can be summarized in the following table.


\begin{center}
{\small
{\bf Table: Properties of homogeneous Lagrangian subspaces of $S^6$}\\
\begin{tabular}{|r|c|c|c|c}
\hline
& $\dim T_{L^3_i}$ & $L_i^3$ regular? & deformation space of $L_i^3$\\
\hline 
\hline
&&& \vspace{-3mm}\\
$i=0$ & $8$ & yes & $\Mm_{d=1} = G_2/SO(4)$\\
&&& \vspace{-3mm}\\
$i=1$ & $10$ & yes & $G_2/U(2)$\\
&&& \vspace{-3mm}\\
$i=2$ & $16$ & yes & $\Mm_{d=2} = G_2 \times_{SU(3)} Sl(3, \C)/SO(3,\C)$\\
&&& \vspace{-3mm}\\
$i=3$ & $41$ & ? &\\
&&& \vspace{-3mm}  \\
$i=4$ & $22$ & ? &\\
\hline
\end{tabular}
}
\end{center}

\begin{remark}\label{rem:killing}
\begin{enumerate}
\item
The rigidity of the Lagrangian sphere $S^3(1)$  also follows  from the  Simons rigidity theorem  which states that each geodesic  sphere  in $S^n$  is rigid  as minimal submanifold   up to the  motion  of the isometry  group $SO(n+1)$ \cite[Theorem 5.2.3]{Simons1968}. 
\item
The term  rigidity  we use  here   corresponds  to the  same notion of rigidity  of \cite[Definition 4.12, p. 28]{Lotay2012}.  Our notation of  regularity  corresponds to  Lotay's  notion of  Jacobi integrability \cite[Definition 3.18, p.18]{Lotay2012}.  Lotay's  notion  of   stability  corresponds to  a special case of our notion of  regularity \cite[Definition 4.12, p. 27-28]{Lotay2012}.
\end{enumerate}
\end{remark}

\section{Appendix. Real analytic  Banach manifolds  and   implicit  function    theorem}

In this Appendix  we  recall the notion  of  a real  analytic Banach manifold, following   Eells \cite{Eells1966}, see also \cite[\S 2]{Douady1966}, and the  analytic   inverse  function theorem, following Douady  \cite[\S 6]{Douady1966}.  Then we     prove a simple criterion for a smooth  mapping  to be analytic (Lemma \ref{lem:analb}). We also derive    the  analytic   implicit  function theorem (Proposition \ref{prop:iftb})  from the  analytic inverse  function theorem.  We always  work over   the    field $\R$ of real numbers, if not   sp ecified  otherwise.

Let $E$ and $F$ be real Banach spaces, and $U$ an open subset of  $E$. Denote by $L(E, F)$  the vector space of all continuous linear maps  $u: E\to F$. Let us recall that  a map $\phi: U \to F$ is
called {\it Fr\'echet differentiable} at $x_0\in U$ if there is an element $\Phi \in L(E, F)$ such that
$$ \lim_{v \to 0}\frac{|\phi(x_0+ v) - \phi(x) - \Phi(v)|_F}{|v|_E} = 0.$$
In this case $\Phi(v)$ is unique  and   also denoted by  $\phi_*(x, v)$ or $d\phi(x; v)$.  We regard $d\phi$ as a  mapping
from $U$ to  $L(E, F)$.

Denote by $SL^r(E, F)$ the  class of continuous  symmetric
$r$-linear  maps $E \times  _{ r \, times} \times E \to F$. Inductively, 
$\phi$ is  of class $C^r$, if $d^r \phi: U \to  SL^r(E, F)$ is  continuous.

\begin{definition}\label{def:anal} Let $E$  and $F$ be  two Banach spaces and $U$   an open  subset in $E$.  A  smooth map
$h : U \to F$ is called  {\it analytic  at a point $a \in U$},   if 
there exists $r > 0$ such that
 for all $|x| < r$ we have $ (a +x) \in U$ and
\begin{equation}
 h(a + x) = \sum_{k=0}^\infty \frac{ d^k h (a; x, \cdots, x)}{k!}. \label{eq:eells}
\end{equation}
\end{definition}

To recognize  analytic    maps  among smooth maps      we   use  the following  Lemma.
 
\begin{lemma}\label{lem:analb}   Let $U$ be an open subset of a  Banach space $E$. A smooth mapping   $f$ from $U $   to a  Banach space  $F$ is   analytic at  a point $x \in U$ iff  there exists a  positive number $r$  depending on $x$ such that the following holds.   For any affine  line  $l$ through  $x$ the restriction of $f$ to  $l\cap U$ is  analytic  at $x$  with   radius of convergence  at least $r$.
\end{lemma}

\begin{proof}  
The ``only if" assertion of Lemma \ref{lem:analb}  is    straightforward.
Now let us   prove the  ``if'' assertion of  Lemma \ref{lem:analb}.
Since    we  do not assume  any condition on $f$, w.l.o.g. we can  assume  that $x =0 \in E$.
By the assumption    the  sphere  $S(r)$ of radius $r$     and  with center at  $x = 0$ lies in  $U$.  Let $  s \in  S(r)$. 
Set $g(t): = f(ts)$.  By the assumption  $g(t)$  is analytic  at $0$  with   the convergence radius at least $r$.
Since  
$$\frac{dg}{dt}   = df (ts; s)$$
and hence
$$\frac{d^n g}{dt^n}   = d^nf (ts; s,  \cdots ,  s )$$
we have  the Taylor    expansion  of  $g$ at zero  is
\begin{equation}
 g(t) = f(0) + \sum_{ k =0} ^ \infty  \frac{d^kf (0, s, \cdots, s)t^k}{k!}.\label{eq:taylor1}
 \end{equation}
 Comparing (\ref{eq:taylor1}) with (\ref{eq:eells}), we  obtain immediately  Lemma \ref{lem:analb}.
\end{proof}

Having  the notion of   an analytic   mapping between  Banach vector  spaces, it  is straightforward to define the notion of  an  analytic Banach manifold.
Now we   formulate  the  analytic  inverse  function theorem  that has been proved  in \cite{Douady1966}.

\begin{proposition}\label{prop:invb} (\cite[Theorem 1]{Douady1966}) Let $X$ and $Y$ be   two  analytic Banach  manifolds and $f : X \to Y$  an analytic map.   Assume that $b = f(a)$ and $T_a f : T_a X \to T_b Y$  is an isomorphism. Then   $f$ is a local isomorphism. 
\end{proposition}

Now we combine the  implicit   function theorem for   Banach spaces  as formulated  in  \cite[Chapter I, Theorem 5.9]{Lang1999}  and the  analytic inverse  function  theorem  to   prove the following.

\begin{proposition}\label{prop:iftb}  Let  $U, V$  be  an  open  sets  in Banach spaces $E$ and $F$  respectively, and let $f : U \times V \to G$  be an analytic mapping.  Let $(a, b) \in U \times V$  and assume that  the restriction of  the differential $Df$ at $(a, b)$  to $(0,  F)\subset E\times F$   to $G$ is   a topological isomorphism.  Let $f(a, b) = 0$. Then there exist a  small neighborhood $U_0$ of $a$ in $U$  and  an  analytic mapping  $g : U_0 \to V$  such that
$$ f (x, g(x)) = 0$$
for all $x \in U_0$. 
\end{proposition}
\begin{proof}  The  proof of Proposition  \ref{prop:iftb}   repeats  the    proof  of  the implicit  function theorem   given in \cite[p. 19]{Lang1999}. It  reduces to  the   analytic inverse function theorem \ref{prop:invb}  by considering      the new  map $\phi: U\times V \to E \times F,\,  (x, y) \mapsto (x, f(x, y))$, so we omit   the detail  of the proof.
\end{proof}

\subsection*{Acknowledgement} A part of this work has been completed while H.V.L. visited the TU Dortmund as a Gambrinus fellow and while both authors visited the Max Planck Institute f\"ur Mathematik in den Naturwissenschaften in Leipzig. We thank both institutions and the Gambrinus foundation for providing excellent working conditions and financial support. We thank Frank Morgan, Jason Lotay,  Lars Sch\"afer, Ji\v ri Van\v zura  and Kotaro Kawai for helpful remarks.


\begin{thebibliography}{99999}
\bibitem{ADN1964}{\sc  S.  Agmon, A. Douglis  and L. Nirenberg},  Estimates near the boundary  for the solutions of elliptic  partial differential  equations satisfying  general boundary  conditions, II, Comm. Pure  Appl.  Math., 17 (1964), 35-92. 
\bibitem{Artin1968}{\sc M. Artin}, On the  solutions of analytic  equations,  Inventiones  Math., 5 (1968), 277-291. 
\bibitem{Baer1993}{\sc C. Baer}, Real Killing Spinors and Holonomy, Comm. Math. Phys. 154 (1993), 509-521.
\bibitem{Butscher2003}{\sc A. Butscher}, Deformations of minimal Lagrangian submanifolds with boundary, Proc. Amer. Math. Soc. 131 (2003), 1953-1964.
\bibitem{Bryant1987}{\sc R. L. Bryant}, Minimal Lagrangian submanifolds of K\"ahler-Einstein manifolds,  Lect. Notes in Math., 1255 (1987), 1-12.
\bibitem{CS2002} {\sc S. Chiossi  and  S. Salamon},  The intrinsic torsion of $SU(3)$ and $G_2$-structures, Differential geometry, Valencia, 2001, 115-133, World Sci. Publ., River Edge, NJ, 2002. 
\bibitem{DVV1990}{\sc F. Dillen, L. Verstraelen and  L. Vranken},  Classification  of totally  real 3-dimensional  submanifolds  of $S^6(1)$  with  $K \ge 1/16$, J. Math. Soc. Japan, 42(1990), 565-584.
\bibitem{DK1981}{\sc D. DeTurck  and J. Kazdan}, Some regularity theorems in Riemannian geometry,  Ann. Scient. Ec. Norm Sup  4. Serie, 14 (1981), 249-260.
\bibitem{DV1996}{\sc  F. Dillen, L. Verstraelen},  Totally real submanifolds of $S^6$ satisfying Chen's Equality, Trans. AMS,  348(1996),  1633-1646.
\bibitem{Douady1966}{\sc A. Douady}, Le probl\'eme des modules pour  les  sous-espaces  analytiques  compacts  d'un espace  analytique donn\'e,  Annales de l'institut 
Fourier, tome 16 (1966), 1-95.
\bibitem{Eells1966}{\sc J. Eells}, A setting for global analysis, Bull. AMS, 72(1966), 751-807.
\bibitem{Ejiri}{\sc N. Ejiri}, Totally real submanifolds in a 6-sphere, Proc. Amer. Math. Soc., 83/4,(1981), 759-763.
\bibitem{Folland1989}{\sc G.B. Folland}, Harmonic analysis of the de Rham complex on the sphere,  J. Reine Angew. Math. 398 (1989), 130-143.
\bibitem{Gray1970}{\sc A. Gray}, Nearly K\"ahler manifolds, JDG  4(1970), 283-309.
\bibitem{Gray1976}{\sc A. Gray},  Structure of nearly K\"ahler manifolds, Math. Ann., 223 (1976), 233-248.
\bibitem{GH1980}{\sc A. Gray  and L. M. Hervella}, The sixteen classes  of almost Hermitian manifold  and their  linear invariants, Ann. Math. Pura App. 123(1980), 35-58.
\bibitem{HL1982}{\sc R. Harvey and H. B. Lawson},  Calibrated geometry, Acta Math. 148(1982), 47-157.
\bibitem{Helgason1978}{\sc S. Helgason},  Differential geometry,  Lie  groups  and    symmetric spaces, Academic Press, (1978).
\bibitem{Hitchin1997}{\sc N. Hitchin},  The moduli space of special  Lagrangian submanifolds,  Ann. Scuola Norm. Sup. Pisa Cl. Sci. (4) 25 (1997), no. 3-4, 503-515 (1998).
\bibitem{Joyce2007}{\sc D. Joyce}, Riemannian holonomy groups and calibrated geometry, Oxford, 2007.
\bibitem{KN1969}{\sc  S. Kobayashi and K.Nomizu},  Foundations of  differential geometry, vol. II,  Intersciences Publishers,New York- London- Sydney  1969.
\bibitem{Kirichenko1977}{\sc V. F. Kirichenko}, K-spaces of maximal rank, Mat. Zametki 22 (1977), 465-476.
\bibitem{KP2013}{\sc S. G. Krantz  and H. R. Parks}, A Primer of Real Analytic Functions, Birkh\"ause, 2002.
\bibitem{Lang1999}{\sc S. Lang}, Fundamentals of Differential Geometry, Springer, 1999.
\bibitem{Le1989}{\sc H. V. L\^e},  Minimal $\Phi$-Lagrangian surfaces  in almost Hermitian manifolds,  Math USSR Sbornik, 67 (1990), 379-391.
\bibitem{Le1990}{\sc H. V. L\^e}, Relative calibration and the problem of stability of minimal surfaces, Lect. Notes in Math., Springer-Verlag, 1990,
v1453, 245-262.
\bibitem{Le1990b}{\sc H.V. L\^e}, Jacobi equations on minimal homogeneous submanifolds in homogeneous Riemannian spaces,
 Funct. Anal. Appl. 24 (1990), no. 2, 125-135.
\bibitem{LF1987}{\sc H.V. L\^e, A. T. Fomenko},  A criterion for the minimality of Lagrangian submanifolds in K\"ahlerian manifolds. 
  Math. Notes, 1987, v.42, 810-816.
\bibitem{LO2012} {\sc  H. V. L\^e  and Y. G. Oh},  Deformations of coisotropic submanifolds in locally  conformal symplectic manifolds, arXiv:1208.3590 (accepted  for Asian J. Math.).
\bibitem{Lichnerowicz1955} {\sc A. Lichnerowicz},  Theorie globale des connexions et des groupes d'holonomie,  Roma, Edizioni Cremonese 1955.
\bibitem{LZ1979}{\sc S. Lojasiewicz, E. Zehnder}, An Inverse Function Theorem in Fr\'echet-Spaces, Jour.Func.Anal 33 (1979), 165-174
\bibitem{Lotay2011}{\sc J.Lotay}, Ruled  Lagrangian  submanifolds of 6-sphere, T.A.M.S. 363(2011), 2305-2339.
\bibitem{Lotay2012}{\sc J. Lotay}, Stability of coassociative conical singularities, Comm. Anal. Geom. 20 (2012), no. 4, 803–867.
\bibitem{Mashimo1985}{\sc K. Mashimo}, Homogeneous totally  real submanifolds  of $S^6$,  Tsukuba J. Math., 9 (1985), 185-202.
\bibitem{McLean1998} {\sc R. McLean},  Deformations of Calibrated submanifolds,  Comm. in Analysis and Geom. 6 (1998), 705-747.
\bibitem{Morrey1954}{\sc C. B. Morrey}, Second order elliptic system of partial differential equations, 101-160,  in Contribution to  the theory
of Partial differential equations, Ann. of Math. Study, 33, Princeton Univ. Press, Princeton, 1954. 
\bibitem{Morrey1958}{\sc C. B. Morrey}, On the Analyticity of the Solutions of Analytic Non-Linear Elliptic Systems of Partial
Differential Equations, I , II, AJM,  vol 80 (1958),198-218  and 219-237. 
\bibitem{Morrey2008}{\sc C. B. Morrey}, Multiple Integrals  in the Calculus of the variations,  Springer 2008.
\bibitem{Morvan1981}{\sc  J. M. Morvan}, Classes de Maslov d'une immersion lagrangienne et minimalite, C. R. Acad. Sci. Paris, (292)1981, 633-636.
\bibitem{Nagy2002} {\sc P. A. Nagy}, On nearly   K\"ahler manifolds,  Ann. Glob.  An. Geom. , 22 (2002), 167-178.
\bibitem{Nagy2002a} {\sc P. A. Nagy}, Nearly  K\"ahler geometry and   Riemannian foliations, Asian  J. Math.,  6(2002), 481-504.
\bibitem{Oh1990}{\sc Y.G. Oh}, Second variation  and  stability  of minimal  lagrangian submanifolds  in K\"ahler  manifolds,  Invent. Math.  101(1990), 501-519.
\bibitem{Ohnita1987}{\sc Y. Ohnita}, On stability of minimal submanifolds in compact symmetric spaces, Compositio Mathematica,  64(1987), 157-189.
\bibitem{OV1990}{\sc A.L. Onishchik, E.B. Vinberg}, Lie Groups and Algebraic Groups, Springer, New York, 1990.
\bibitem{OP2005} {\sc Y. G. Oh and  J. S. Park}, Deformations of coisotropic submanifolds and  strong homotopy Lie algebroids,  Invent. Math.  161(2005),  287-360. 
\bibitem{Palmer1998}{\sc B.Palmer}, Calibrations and Lagrangian submanifolds
in the six  sphere,  Tohoku Math. J.  50(1998), 303-315.
\bibitem{Rosenberg1997} {\sc Rosenberg, S}, The Laplacian on a Riemannian manifold: an introduction to analysis on manifolds, Cambridge Univ. Press (1997)
\bibitem{SW2003}{\sc R. Schoen, J. Wolfson},  The volume functional for Lagrangian submanifolds. Lectures on partial differential equations, 181-191, New Stud. Adv. Math., 2, Int. Press, Somerville, MA, 2003.
\bibitem{Simons1968} {\sc J. Simons}, Minimal varietes in Riemannian manifolds,  Annals of Math.,  88(1968), 62-105.
\bibitem{SS2010} {\sc L. Sch\"afer  and K. Smoczyk}, Decomposition and minimality of Lagrangian submanifolds in nearly
K\"ahler  manifolds, Annals of Global Analysis and Geometry
37 (2010), 221-240.
\end{thebibliography}
\end{document}